\documentclass{siamltex}
\usepackage{amssymb}
\usepackage{bm}
\usepackage{graphicx}
\usepackage{graphics}
\usepackage{caption2}
\usepackage{psfrag}
\usepackage{arydshln}
\usepackage{amsmath}
\usepackage{amscd}
\usepackage{amsfonts}
\usepackage{float}
\usepackage{latexsym}
\usepackage{color}
\usepackage{multirow}
\usepackage{lscape}
\usepackage{arydshln}
\usepackage{epstopdf}
\usepackage{lineno}

\newtheorem{thm}{Theorem}[section]
\newtheorem{cor}[thm]{Corollary}
\newtheorem{lem}[thm]{Lemma}
\newcommand{\om}{\Omega}
\def\C{\mathbb{C}}
\newcommand{\p}{\partial}
\def\d{\mathrm{d}}

\def\dt{\p_t}

\newcommand{\ds}{\displaystyle}
\newcommand{\f}{\frac}

\begin{document}

\title{Local modification of subdiffusion by initial Fickian diffusion: Multiscale modeling, analysis and computation}
\author{
Xiangcheng Zheng\thanks{School of Mathematics, Shandong University, Jinan 250100, China}
\and
Yiqun Li\thanks{Department of Mathematics, University of South Carolina, Columbia, SC 29208, USA}
\and
Wenlin Qiu\thanks{School of Mathematics, Shandong University, Jinan 250100, China}
}
\maketitle

\begin{abstract}
We propose a local modification of the standard subdiffusion model by introducing the initial Fickian diffusion, which results in a multiscale diffusion model. The developed model resolves the incompatibility between the nonlocal operators in subdiffusion and the local initial conditions and thus eliminates the initial singularity of the solutions of the subdiffusion, while retaining its heavy tail behavior away from the initial time. The well-posedness of the model and high-order regularity estimates of its solutions are analyzed by resolvent estimates, based on which the numerical discretization and analysis are performed. Numerical experiments are carried out to substantiate the theoretical findings.
\end{abstract}



\pagestyle{myheadings}
\thispagestyle{plain}

\markboth{}{Local modification of subdiffusion}
\section{Introduction}
We consider the subdiffusion model  with the variable exponent \cite{FanHu,GarGiu,JinLiZouSINUM,LiLuo,LiLiZha,ZenZhaKar,ZhuLiu}, which has been used in various fields such as the transient dispersion in heterogeneous media \cite{SunZha} (more applications could be found in the review \cite{SunCha})
\begin{equation}\label{VtFDEs}\begin{array}{c}
\p_t u(\bm x,t)- \p_t^{\alpha(t)} \Delta u (\bm x,t)= f(\bm x,t) ,~~(\bm x,t) \in \Omega\times(0,T]; \\ [0.05in]
\ds u(\bm x,0)=u_0(\bm x),~\bm x\in \Omega; \quad u(\bm x,t) = 0,~(\bm x,t) \in \p \Omega\times[0,T].
\end{array}\end{equation}
Here $\Omega \subset \mathbb{R}^d$  is a simply-connected bounded domain with the piecewise smooth boundary $\p \om$ with convex corners, $\bm x := (x_1,\cdots,x_d)$ with $1 \le d \le 3$ denotes the spatial variables, $f$ and $u_0$ refer to the source term and the initial value, respectively, and $\p_t^{\alpha(t)}$ with the variable exponent $0\leq \alpha(t)<1$ is the Riemann-Liouville fractional differential operator defined by the symbol $*$ of convolution \cite{LorHar}
\begin{align}\label{var_RL}
&\ds \p_t^{\alpha(t)}u :=\p_t\, I_t^{1-\alpha(t)}u, \qquad I_t^{1-\alpha(t)}u : =\Big(\frac{t^{-\alpha(t)}}{\Gamma(1-\alpha(t))}\Big)* u(\bm x,t).
\end{align}
\subsection{Modeling issues}
In this work, we explore more benefits and properties of (\ref{VtFDEs}) via rigorous analysis.  A key ingredient lies in that we require $\alpha(0)=0$, that is, we expect the model (\ref{VtFDEs}) to degenerate to its integer-order analogue at the time instant $t=0$ in order to resolve the incompatibility between the nonlocal operators in subdiffusion and the local initial conditions, as mentioned in \cite{JMAA}. If such local modification works, the model (\ref{VtFDEs}) indeed exhibits multiple diffusion scales, i.e. the Fickian diffusion near the initial time and the subdiffusion away from the initial time.

 To give an intuitive motivation, we compare the solutions of the multiscale diffusion model (\ref{VtFDEs}) with its integer-order analogue, i.e. model (\ref{VtFDEs}) with $\alpha(t)\equiv 0$, which models the purely Fickian diffusion
$$\p_t u(\bm x,t)- \Delta u (\bm x,t)= f(\bm x,t)$$
 and its constant-exponent analogue, i.e. model (\ref{VtFDEs}) with $\alpha(t)\equiv \bar \alpha$ for some $0<\bar\alpha<1$, which accurately captures the power-law decaying property of the purely subdiffusion through heterogeneous porous media \cite{MeeSik,MetKla} and has attracted extensive research activities \cite{ChuEfeCMAME,DenLi,DenWan,JiaOu,LinTar,SakYam,StyOriGra}
\begin{equation}\label{tFDE}
\p_t u(\bm x,t)- \p_t^{\bar\alpha} \Delta u (\bm x,t)= f(\bm x,t).
\end{equation}
The solution curves of the above three models are presented in Figure \ref{fig1}, from which we observe that:
\begin{itemize}
\item In comparison with the solution curves of the purely Fickian diffusion and subdiffusion, the solutions to model (\ref{VtFDEs}) transit from the initial Fickian diffusion behavior to the long-term subdiffusion behavior, which indicates the multiscale feature of this model.

\item Compared with the initial rapid change (i.e. the initial singularity) of the solutions to the purely subdiffusion, the solutions to the multiscale diffusion model (\ref{VtFDEs}) exhibits the Fickian diffusion behavior near the initial time, which eliminates the incompatibility between the nonlocal operators in subdiffusion and the local initial conditions.

\item Compared with the exponential decay of the solutions to the purely Fickian diffusion away from the initial time, the solutions to the multiscale diffusion model (\ref{VtFDEs}) exhibits the subdiffusion behavior, which retains the advantages of the subdiffusion in modeling the power law decay and heavy tail behavior.

\end{itemize}

In summary, the multiscale diffusion model (\ref{VtFDEs}) provides a framework incorporating both Fickian and subdiffusion models and simultaneously retaining their features and advantages, which indicates the superiority of model (\ref{VtFDEs}) and motivates the current study on its mathematical and numerical analysis.

\begin{figure}[h!]
	\setlength{\abovecaptionskip}{0pt}
	\centering	
\includegraphics[width=3in,height=2in]{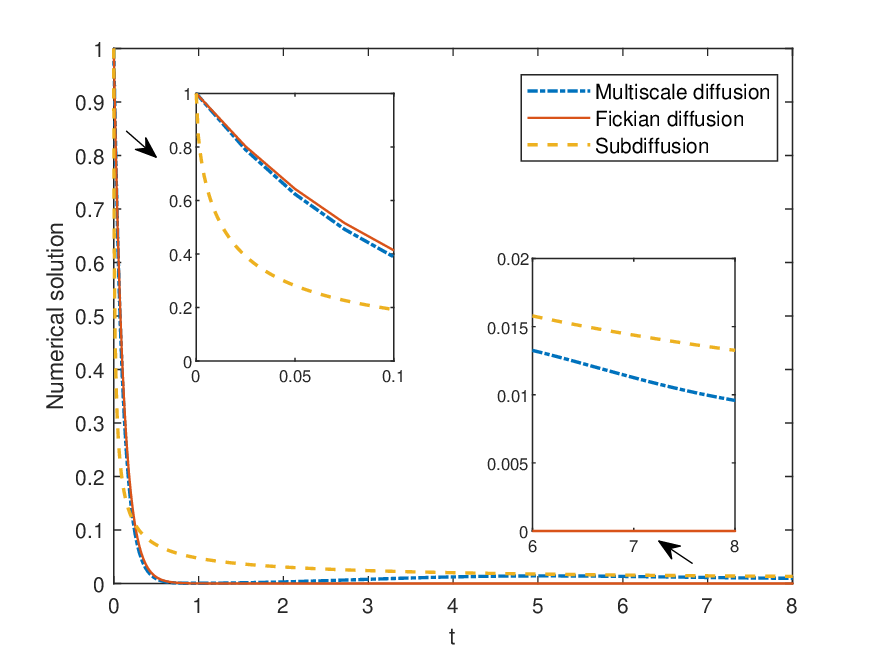}
	\caption{Plots of the solutions $u(0.5,t)$ for three models under $\Omega=(0,1)$, $T=8$, $f\equiv 0$, $u_0=\sin(\pi x)$. Here the variable exponent is chosen as $\alpha(t)=\alpha(T) + (\alpha(0) - \alpha(T))\big(1 - \frac{t}{T} -\frac{\sin(2\pi (1-t/T))}{2\pi}  \big)$,  a smooth and monotonic function on $[0,T]$ with end values $\alpha(0) = 0$ and $\alpha(T) = 0.4$. We accordingly fix $\bar \alpha=0.4$.  }
	\label{fig1}
\vspace{-0.15in}
\end{figure}

\subsection{Analysis issues}
There exist rare investigations for analysis of variable-exponent fractional partial differential equations in the literature. Recently, there are some progresses on the mathematical and numerical analysis of variable-exponent mobile-immobile time-fractional diffusion equations \cite{LiWanWan,LiWanZhe,JMAA,ZheWanIMA}, a different model from (\ref{VtFDEs})
\begin{equation}\label{mtFDE}
\p_t u(\bm x,t)+k\,\p_t^{\alpha(t)}u(\bm x,t)-  \Delta u (\bm x,t)= f(\bm x,t).
\end{equation}
Though different methods have been developed to study (\ref{mtFDE}), all derivations are based on a critical fact that the $\p_t u(\bm x,t)$ in (\ref{mtFDE}) dominates the $\p_t^{\alpha(t)}u(\bm x,t)$ such that the later could be considered as a low-order or perturbation term. In this case, the effects of variable-exponent fractional operators, which do not have nice properties as constant-exponent operators, are weakened. However, this phenomenon does not appear in (\ref{VtFDEs}) due to the composition of temporal and spatial operators, which brings essential differences and difficulties in comparison with (\ref{mtFDE}).

In this work, we derive an equivalent model as (\ref{VtFDEs}), which paves the way to perform analysis. Based on this more feasible formulation, we prove the well-posedness and regularity of the solutions, as well as developing and analyzing the fully-discrete finite element scheme. The rest of the paper is organized as follows: In \S \ref{S:Pre}, we introduce the notations and preliminary results to be used subsequently. In \S \ref{S:Well}, we prove the well-posedness of the multiscale diffusion model and the regularity estimates of its solutions. In \S \ref{S:Hreg}, we prove high-order regularity of the solutions via resolvent estimates. In \S \ref{S:Nume}, we derive a fully-discrete finite element scheme to the model and prove its error estimates. Numerical experiments are performed in \S \ref{S:Experiment} to substantiate the theoretical findings.

\section{Preliminaries} \label{S:Pre} We present preliminaries to be used subsequently.

\subsection{Notations}
Let $L^p(\om)$ with $1 \le p \le \infty$ be the Banach space of $p$th power Lebesgue integrable functions on $\om$. For a positive integer $m$,
let  $ W^{m, p}(\Omega)$ be the Sobolev space of $L^p$ functions with $m$th weakly derivatives in $L^p(\om)$ (similarly defined with $\om$ replaced by an interval $\mathcal I$). Let  $H^m(\Omega) := W^{m,2}(\Omega)$ and $H^m_0(\Omega)$ be its subspace with the zero boundary condition up to order $m-1$. For a non-integer $s\geq 0$, $H^s(\Omega)$ is defined via interpolation \cite{AdaFou}. Let $0 < \lambda_i \uparrow \infty$ be the eigenvalues and $\{\phi_i\}_{i=1}^\infty$ be the corresponding orthonormal eigenfunctions of the problem $-\Delta \phi_i = \lambda_i \phi_i$ with the zero boundary condition. We introduce the Sobolev space $\check{H}^s(\Omega)$ for $s\geq 0$ by
$$\ds \check{H}^{s}(\Omega) := \bigg \{ v \in L^2(\Omega): \| v \|_{\check{H}^s}^2 : = \sum_{i=1}^{\infty} \lambda_i^{s} (v,\phi_i)^2 < \infty \bigg \},$$
which is a subspace of $H^s(\Omega)$ satisfying $\check{H}^0(\Omega) = L^2(\Omega)$ and $\check{H}^2(\Omega) = H^2(\Omega) \cap H^1_0(\Omega)$ \cite{Tho}.
For a Banach space $\mathcal{X}$, let $W^{m, p}(0,T; \mathcal{X})$ be the space of functions in $W^{m, p}(0,T)$ with respect to $\|\cdot\|_{\mathcal {X}}$.  All spaces are equipped with standard norms \cite{AdaFou,Eva}.

We use $Q$, $Q_i$ and $Q_*$ to denote generic positive constants in which $Q$ may assume different values at different occurrences. We set $\|\cdot\|:=\|\cdot\|_{L^2(\Omega)}$ and $L^p(\mathcal X)$ for $L^p(0,T;\mathcal X)$ for brevity, and drop the notation $\om$ in the spaces and norms if no confusion occurs.

We make  the {\it Assumption A} throughout this paper:
\begin{itemize}
\item[(a)]  $0 \le \alpha( t) \le \alpha^* <1$ and $|\alpha^{\prime}(t)| $, $|\alpha^{\prime \prime}(t)| \le Q_*$ on $[0, T]$. In addition,    $\alpha( 0) = 0$.
\item[(b)] $f \in L^p(L^2)$ for $1< p <\infty$ and $ u_0 \in \check H^2$.
\end{itemize}


\subsection{Solution representation and resolvent estimates}

For $\theta\in(\pi/2,\pi)$ and $\delta > 0$, let $\Gamma_\theta$ be the contour in the complex plane defined by
\begin{align*}
\Gamma_\theta := \big \{z\in\C: |{\rm arg}(z)|=\theta, |z|\ge \delta \big \}
\cup \big \{z \in\C: |{\rm arg}(z)|\le \theta, |z|= \delta \big \}.
\end{align*}
The following inequalities hold for  $0<\mu\leq 1$ and $Q=Q(\theta,\mu)$ \cite{Akr,Lub}
\begin{equation}\label{GammaEstimate}
\int_{\Gamma_\theta} |z|^{\mu-1} |e^{tz}|  \, |d z| \le Q t^{-\mu},~~
\bigg \| \int_{\Gamma_\theta} z^\mu (z-\Delta)^{-1} e^{tz}
\, d z \bigg \|_{L^2\rightarrow L^2} \le \f{Q}{t^{\mu}}, ~~t\in (0,T].
\end{equation}

For any $q \in L^1_{loc}(\mathcal I)$, the Laplace transform $\mathcal L$ of its extension $\tilde q(t)$ to zero outside $\mathcal I$ and the corresponding inverse transform $\mathcal L^{-1}$ are denoted by
\begin{equation}\begin{array}{l}\label{Laplace}
\ds \mathcal{L}q(z):=\int_0^\infty \tilde q(t)e^{-tz}d t, \qquad \mathcal{L}^{-1}(\mathcal Lq(z)):=\frac{1}{2\pi \rm i}\int_{\Gamma_\theta} e^{tz}\mathcal Lq(z)d z=q(t),\\[0.15in]
\ds \mathcal L\big ({}^R\partial_t^\gamma q(t) \big ) = z^\gamma \mathcal L\big (q(t) \big),  \quad {}^R\partial_t^\gamma q :=  \p_t \,I_t^{1-\gamma} q, \qquad 0 \leq \gamma<1.
\end{array}
\end{equation}

Let $E(t):=e^{t\Delta}$ be the semigroup generated by the Dirichlet Laplacian. The solution $u$ to the heat equation
\begin{equation}\label{HeatPDE}\begin{array}{c}
\ds \p_t u(\bm x,t) - \Delta u(\bm x,t)  = f(\bm x,t),~~(\bm x,t)\in\Omega\times(0,T], \\[0.05in]
u(\bm x,t)=0, \quad (\bm x,t)\in \partial\Omega\times(0,T], ~~u(\bm x,0)=0,~~\bm x\in\Omega
\end{array}\end{equation}
can be expressed in terms of the $E(t)$ via the Duhamel's principle
\begin{equation}\label{HeatSoln}
u(\bm x,t) = \int_0^{t} E(t-\theta) f(\bm x,\theta) d\theta,
\end{equation}
where $E(t)$ has the spectral decomposition and expression in terms of the inverse Laplace transform
\begin{equation}\label{Et:Express}
E(t)\psi(\bm x) =\sum_{i=1}^\infty e^{-\lambda_i t}(\psi,\phi_i)\phi_i(\bm x)= \frac{1}{2\pi {\rm i}}\int_{\Gamma_\theta}e^{zt}  (z-\Delta)^{-1}\psi(\bm x) \, d z, ~~\psi \in L^2(\om).
\end{equation}
The following estimates hold for $s\geq r\geq -1$ and for any $t>0$ \cite{Akr,Lub,Tho}
\begin{equation}\begin{array}{l}\label{E:est}
\hspace{-0.1in} \ds \|E(t)\|_{L^2 \to L^2}\le Q,\quad \ds \|E(t) \psi \|_{\check{H}^s}\leq Qt^{-(s-r)/2} \| \psi \|_{\check{H}^r}, \qquad \psi \in \check{H}^r.
\end{array}
\end{equation}

\begin{lem}\label{lem:Lp} \cite{Akr}
If $f\in L^p(L^2)$ for $1<p<\infty$, problem (\ref{HeatPDE}) has a unique solution $u\in W^{1,p}(L^2)\cap L^p(\check{H}^2)$ given by (\ref{HeatSoln}) such that
\begin{equation*}\label{HeatPDEestimate}
\|u\|_{W^{1,p}(0,t;L^2)} + \|u\|_{L^p(0,t;\check{H}^2)} \le Q \| f\|_{L^p(0,t;L^2)}, ~~0<t\leq T,
\end{equation*}
where $Q$ is independent of $f$, $t$ or $T$.
\end{lem}

We finally apply \eqref{Lop:e1} to introduce the following relations which will be used frequently in subsequent analysis \cite{Pod}
\begin{equation}\label{Model:e5}\begin{array}{l}
\ds (t-s)^{-\alpha( t-s)} =  e^{-\alpha( t-s)\ln(t-s)}\le Q,\\[0.075in]
\ds \int_{0}^{T} e^{-\gamma t} t^{-\mu} d t \leq \gamma^{\mu-1} \int_{0}^{\infty} e^{-s} s^{-\mu} ds =\gamma^{\mu-1} \Gamma(1-\mu),\quad \mu < 1.
\end{array}\end{equation}

\section{Well-posedness}\label{S:Well}
\subsection{Model reformulation}
We give an equivalent but more feasible formulation of the multiscale diffusion model (\ref{VtFDEs}) to facilitate the analysis.
\begin{thm}\label{Model}
  If $u$ solves the multiscale diffusion model \eqref{VtFDEs}, then $u$ solves the following system
  \begin{equation}\label{VtFDEs1}\begin{array}{c}
\ds \p_t u(\bm x,t)-  \Delta u (\bm x,t)= f(\bm x,t)+B(t) u(\bm x,t) ,~~(\bm x,t) \in \Omega\times(0,T]; \\ [0.05in]
\ds u(\bm x,0)=u_0(\bm x),~\bm x\in \Omega; \quad u(\bm x,t) = 0,~(\bm x,t) \in \p \Omega\times[0,T]
\end{array}\end{equation}
where
$$ B(t) u := g(t)* \Delta u\text{ with }g( t) := \p_t\Big(\frac{t^{-\alpha(t)}}{\Gamma(1-\alpha(t))}\Big).$$ 
Conversely, if $u$ solves \eqref{VtFDEs1}, then $u$ solves the multiscale diffusion model  \eqref{VtFDEs}.
\end{thm}

\begin{proof}
We first prove the problem \eqref{VtFDEs} could be reformulated to \eqref{VtFDEs1}.
We rewrite the kernel in \eqref{var_RL}  as
\begin{equation}\label{IntKer}
  \ds \f{(t-s)^{-\alpha( t-s)}}{\Gamma(1-\alpha( t-s))} =\int_s^t \p_y \Big(\f{(y-s)^{-\alpha( y-s)}}{\Gamma(1-\alpha( y-s))}\Big) dy + \f{(y-s)^{-\alpha( y-s)}}{\Gamma(1-\alpha( y-s))} \bigg|_{y=s},
\end{equation}
where the last term could be evaluated by {\it Assumption A} (cf. \S 2) and l'Hospital's rule as follows
\begin{equation}\begin{array}{l}\label{Lop:e1}
\ds \lim_{t\rightarrow 0^+}\alpha(t)\ln(t)= \lim_{t\rightarrow 0^+} \f{\alpha(t)}{({1}/{\ln(t)})} =\lim_{t\rightarrow 0^+} - t \ln^2(t)  \alpha^\prime(t) =0,\\[0.15in]
  \ds \lim_{y\rightarrow s^+}\f{(y-s)^{-\alpha( y-s)}}{\Gamma(1-\alpha( y-s))}=\lim_{y\rightarrow s^+} \f{e^{-\alpha(y-s)\ln(y-s)}}{\Gamma(1-\alpha( y-s))}=1.
  \end{array}
\end{equation}
We then incorporate  \eqref{IntKer}--\eqref{Lop:e1} into \eqref{var_RL} to reformulate $\p_t^{\alpha(t)} \Delta u$ in \eqref{VtFDEs}  and  interchange the order of integration to obtain
\begin{equation}\begin{array}{cl}\label{Leib:e1}
 \ds \p_t^{\alpha(t)}\Delta u & \hspace{-0.1in} \ds = \p_t \int_0^t \f{(t-s)^{-\alpha( t-s)}}{\Gamma(1-\alpha( t-s))} \Delta u(\bm x,s) ds\\ [0.15in]
  &\hspace{-0.1in} \ds =\p_t  \int_0^t \int_s^t \p_y \Big(\f{(y-s)^{-\alpha( y-s)}}{\Gamma(1-\alpha( y-s))}\Big) \Delta u(\bm x,s) dyds + \p_t  \int_0^t  \Delta u(\bm x,s) ds\\[0.15in]
 \ds &\hspace{-0.1in} \ds =\p_t \int_0^t \int_0^y \p_y \Big(\f{(y-s)^{-\alpha( y-s)}}{\Gamma(1-\alpha( y-s))}\Big) \Delta u(\bm x,s) ds dy   +  \Delta u(\bm x,t) \\[0.15in]
 \ds &\hspace{-0.1in} \ds =\int_0^t \p_t \Big(\f{(t-s)^{-\alpha( t-s)}}{\Gamma(1-\alpha( t-s))}\Big) \Delta u(\bm x,s) ds   +  \Delta u(\bm x,t) \\[0.15in]
  \ds &\hspace{-0.1in} \ds =g(t)* \Delta u(\bm x,t) + \Delta u(\bm x,t),
\end{array}
\end{equation}
by which the multiscale diffusion model  \eqref{VtFDEs} could be rewritten as \eqref{VtFDEs1}, and thus we prove the first statement of this theorem.
Conversely, suppose $u$ is the solution to \eqref{VtFDEs1}, it also solves \eqref{VtFDEs} by the equality  \eqref{Leib:e1}. We thus complete the proof of the theorem.
\end{proof}

\subsection{Auxiliary estimates}
We introduce the following lemmas for future use.

\begin{lem}\label{Sol:Lem0}
Suppose Assumption A holds. Then there exists a positive constant $Q = Q(\alpha^*, Q_*,T)$ such that the following estimates hold:

Case 1: If $\alpha^\prime( 0)\neq 0$, we have
\begin{equation}\label{Lem0:est:e0}
\ds \big| g( t) \big| \le Q |\ln(t)|,\quad  \big|  g^\prime( t) \big| \le Q t^{-1}.
\end{equation}

Case 2: If $\alpha^\prime( 0) = 0$ and $ \alpha^{\prime \prime}( 0) \neq 0$,  we have
\begin{equation}\label{Lem0:est:e1}
\ds \big| g( t) \big| \le Q ,\quad  \big| g^\prime( t) \big| \le Q  |\ln(t)|.
\end{equation}

Case 3: If $  \alpha^{\prime }( 0) = \alpha^{\prime \prime}( 0)= 0$ and $\lim_{t \rightarrow 0^+}  \alpha^{\prime \prime}(t) \ln(t)$ exists,  we have
\begin{equation}\label{Lem0:est:e2}
\ds \big| g(t) \big| \le Q ,\quad  \big| g^\prime( t) \big| \le Q .
\end{equation}

\end{lem}
\begin{proof}
We  first prove Case 1. For the convenience of analysis, we split $g( t)$  as
\begin{equation}\begin{array}{l}\label{Lem0:e1}
\ds g(t) = \f{t^{-\alpha( t)}}{\Gamma(1-\alpha( t))} G( t),\quad
\ds G( t)= - \alpha^\prime( t) \ln(t) -\f{\alpha( t)}{t} + \f{\Gamma^\prime(1-\alpha( t))\alpha^\prime( t)}{\Gamma(1-\alpha( t))}
\end{array}
\end{equation}
with $\big|\f{t^{-\alpha( t)}}{\Gamma(1-\alpha( t))}  \big| \le Q $ by the first estimate in \eqref{Model:e5}. The second term in $G( t)$ could be bounded as follows
\begin{equation}\begin{array}{l}\label{Lem0:e2}
\hspace{-0.15in}\ds \f{\alpha( t)}{t} = \f{\alpha( t)-\alpha( 0)}{t} = \alpha^\prime( 0) + \f{1}{t} \int_0^t  \alpha^{\prime \prime}( y) (t-y) dy,\quad
\ds \bigg|\f{\alpha( t)}{t}\bigg|   \le  Q Q_*,
\end{array}
\end{equation}
which, together with \eqref{Lem0:e1}, yields $|G( t) | \le Q|\ln(t)|$ and thus gives the first estimate in \eqref{Lem0:est:e0}.
In addition, we employ the  expressions in \eqref{Lem0:e1} to evaluate $ g^\prime( t)$ as
\begin{equation}\begin{array}{l}\label{Lem0:e3}
 \hspace{-1in} \ds  g^\prime( t)  \ds = \f{t^{-\alpha( t)}}{\Gamma(1-\alpha( t))} \Big(G^2( t) +  G^\prime(t)\Big),\\[0.1in]
 \hspace{-1in} \ds  G^\prime(t) \ds =  - \alpha^{\prime\prime}( t) \ln(t) - \f{\alpha^{\prime}( t)}{t} -  \Big(\f{\alpha( t)}{t} \Big)^{\prime} \ds     +   \Big(\f{\Gamma^\prime(1-\alpha( t))\alpha^\prime( t)}{\Gamma(1-\alpha( t))}\Big)^\prime.
  \end{array}
\end{equation}
We incorporate the estimate
$$\ds \bigg|\Big(\f{\alpha( t)}{t} \Big)^\prime \bigg| \le \bigg| \f{1}{t}\int_0^t  \alpha^{\prime \prime}( y) dy\bigg| + \bigg| \f{1}{t^2}\int_0^t  \alpha^{\prime \prime}( y) (t-y) dy\bigg| \le Q Q_*$$
 with \eqref{Model:e5} and  $|G( t) | \le Q |\ln(t)|$ to bound $g^\prime( t)$ in \eqref{Lem0:e3} by
\begin{equation}\begin{array}{l}\label{Ptg}
  \ds |g^\prime( t)| \le Q \big(\big|G^2( t)\big | + \big|G^\prime(t)\big|\big) \le Q Q_*\big(\ln^2(t)+| \ln(t)| +{t}^{-1} \big)  \le Q t^{-1}
  \end{array}
\end{equation}
and thus we arrive at the second estimate in \eqref{Lem0:est:e0}.

Now we prove Case 2. We employ  l'Hospital's rule to evaluate
\begin{equation}\begin{array}{l}\label{Lem0:lop1}
\hspace{-0.15in}\ds \lim_{t\rightarrow 0^+} \alpha^\prime(t)\ln(t)\!=\! \lim_{t\rightarrow 0^+} \f{ \alpha^\prime(t)}{({1}/{\ln(t)})} \!=\!\lim_{t\rightarrow 0^+} - t \ln^2(t)  \alpha^{\prime \prime}(t) \!=\!0,\,\,
\ds \ds \lim_{t\rightarrow 0^+} \f{  \alpha^\prime( t)}{t} \! =\!   \alpha^{\prime \prime}( 0),
\end{array}
\end{equation}
and thus we combine the first estimate in  \eqref{Lem0:lop1}  with \eqref{Lem0:e1}--\eqref{Lem0:e2} to bound $|G( t)| \le Q$ and thus to obtain the first estimate in \eqref{Lem0:est:e1}.  We then combine the second estimate in \eqref{Lem0:lop1}
to find that the most singular term in  $G ^\prime(t)$ in \eqref{Lem0:e3} now becomes $ \alpha^{\prime \prime}( t) \ln(t)$, and thus we derive the second estimate in   \eqref{Lem0:est:e1}.

For Case 3, we combine \eqref{Lem0:e1}, \eqref{Lem0:e3} with estimates in \eqref{Lem0:lop1} to derive the estimates in \eqref{Lem0:est:e2} and thus we complete the proof of the lemma.
\end{proof}

For $\gamma \ge 1$ and $1 < p < \infty$, let $\mathcal X_{\gamma}$ be the space $\mathcal X := \big \{g\in L^{p}(\check H^2): g(0)=0 \big \}$ with the equivalent norm $\|g\|_{\mathcal X_{\gamma}} := \big \|e^{-\gamma t}\Delta g \big \|_{L^p(L^2)}$ \cite{LiWanWan}. Then we prove an auxiliary estimate for the operator $B(t)$.

\begin{lem}\label{Sol:Lem1}
For $\gamma \ge 1$ and $1 < p < \infty$, there exists a positive constant $Q = Q(\alpha^*,T)$ such that
\begin{equation}\label{Lem1:e0}
\bigg \| e^{-\gamma t}\int_0^t   E^\prime(t-s) B(s) v ds \bigg \|_{L^p(L^2)} \le Q \gamma^{\alpha^*-1} \| v \|_{\mathcal{X}_\gamma}.
\end{equation}
\end{lem}

\begin{proof}  For $0 < \varepsilon \ll 1$, we use \eqref{Et:Express} and Laplace transform to  evaluate
\begin{equation}\label{Lem1:Lap} \begin{array}{ll}
& \hspace{-0.3in} \ds \mathcal{L} \bigg[  \int_0^t  E^{\prime}(t-s) B(s) v(\bm x, s) ds \bigg] \\[0.1in]
& \hspace{-0.1in} \ds =\mathcal{L}\bigg[   \int_0^t \p_t \bigg( \f{1}{2\pi \mathrm {i}} \int_{\Gamma_{\theta}} e^{z(t-s)} (z+A)^{-1}dz\bigg)  \big(  B(s) v(\bm x, s)\big)ds \bigg] \\[0.15in]
 &  \hspace{-0.15in} \ds = \mathcal{L}\bigg[   \int_0^t  \bigg( \f{1}{2\pi \mathrm{i}} \int_{\Gamma_{\theta}} e^{z(t-s)} z (z+ A)^{-1}dz\bigg)  (B(s) v(\bm x, s))  ds \bigg] \\[0.15in]
 &  \hspace{-0.15in} \ds = \mathcal{L}\bigg[   \f{1}{2\pi \mathrm{i}} \int_{\Gamma_{\theta}} e^{zt} z (z+ A)^{-1}dz\bigg]  \mathcal{L} \big( B(t) v(\bm x,t) \big)\\[0.15in]
 &  \hspace{-0.15in} \ds = \big(z^{1-\varepsilon} (z+ A)^{-1}\big) \bigl (z^{\varepsilon} \mathcal{L} \big(B(t) v(\bm x,t)  \big ).
\end{array}\end{equation}
Take the inverse Laplace transform of \eqref{Lem1:Lap} and use \eqref{Laplace} to obtain
\begin{equation}\label{Lem1:Inv}\begin{array}{l}
\ds \hspace{-0.1in} \int_0^t E^\prime(t-s) B(s) v(\bm x, s) ds \\[0.1in]
\ds  = \mathcal{L}^{-1} \bigg[\big(z^{1-\varepsilon} (z+A)^{-1}\big) \big(z^{\varepsilon} \mathcal{L} (B(t) v(\bm x,t))\big ) \bigg] \\[0.1in]
\ds = \mathcal{L}^{-1} \bigg[\big(z^{1-\varepsilon} (z+A)^{-1}\big)\bigg] * \mathcal{L}^{-1} \Big[ z^{\varepsilon} \mathcal{L}\big ( B(t) v(\bm x,t)) \big ) \Big] \\[0.1in]
\ds = \int_0^t \bigg[ \f{1}{2\pi \mathrm{i}} \int_{\Gamma_{\theta}} e^{z(t-s)} z^{1-\varepsilon}(z+ A)^{-1}dz\bigg] {}^{R} \p_s^{\varepsilon} \big (  B(s) v(\bm x,s) \big) ds.
\end{array}\end{equation}
Apply the second estimate in \eqref{GammaEstimate} to bound the term in the square bracket by
\begin{equation}\label{Lem1:e1}
 \bigg\|\f{1}{2\pi \mathrm{i}} \int_{\Gamma_{\theta}} e^{z(t-s)} z^{1-\varepsilon}(z+A)^{-1}dz\bigg\|_{L^2 \rightarrow L^2} \le Q (t-s)^{-(1-\varepsilon)}.
\end{equation}
Evaluate the remaining term in the last integral of \eqref{Lem1:Inv} and interchange the order of integration to obtain
\begin{equation}\label{Lem1:e2}\begin{array}{l}
\ds {}^{R}\p_t^{\varepsilon} \big (B(t) v(\bm x,t) \big) 
=\f{\p_t}{\Gamma(1-\varepsilon)}\int_0^t \f{1}{(t-\iota)^{\varepsilon}}\int_0^\iota  g(\iota -s )\Delta v(\bm x, s) ds d\iota\\[0.15in]
\ds \qquad  \qquad \qquad \qquad= \p_t \int_0^t \f{\Delta v(\bm x,s)}{\Gamma(1-\varepsilon)}\int_s^t (t-\iota)^{-\varepsilon} g(\iota-s )d\iota ds\\[0.15in]
\ds \qquad  \qquad \qquad \qquad=  \int_0^t \f{\Delta v(\bm x,s)}{\Gamma(1-\varepsilon)} \p_t \int_s^t (t-\iota)^{-\varepsilon} g(\iota-s )d\iota ds,
\end{array}\end{equation}
where we have used fact  derived from the most singular estimate  for $g$  in Lemma \ref{Sol:Lem0}, i.e., $|g(t)| \le Q \big|\ln(t)\big|$  that
\begin{equation*}\begin{array}{l}
  \ds \bigg|\int_s^t (t-\iota)^{-\varepsilon} g(\iota-s )d\iota \bigg| \le   \ds \bigg|\int_s^t (t-\iota)^{-\varepsilon} \big|\ln(\iota-s )\big|d\iota \bigg|\\[0.1in]
   \ds \le Q \ds \bigg|\int_s^t (t-\iota)^{-\varepsilon} (\iota-s )^{-\alpha^*}d\iota \bigg|
  \ds   =Q (t-s)^{1-\varepsilon-\alpha^*} \rightarrow 0, \quad \mbox{as} \,\,  s \rightarrow t^-.
  \end{array}
\end{equation*}
We introduce the variable substitution $z = (\iota-s)/(t-s)$, which implies
\begin{equation}\label{Lem1vo:sub}
  \iota-s=(t-s)z, \quad t-\iota=(t-s) (1-z), \quad d \iota =(t-s) d z,
\end{equation}
to reformulate the inner integral of \eqref{Lem1:e2} as
$$
  \ds \int_s^t (t-\iota)^{-\varepsilon} g(\iota-s )d\iota = \int_0^1 (t-s)^{1-\varepsilon}(1-z)^{-\varepsilon} g( (t-s)z)dz.
  $$
Differentiate the above expression to derive
\begin{equation}\begin{array}{l}\label{Lem1:e3}
\ds \p_t \int_s^t (t-\iota)^{-\varepsilon} g(\iota-s )d\iota ds = \p_t \int_0^1 (t-s)^{1-\varepsilon}(1-z)^{-\varepsilon} g( (t-s)z)dz\\[0.1in]
\ds \qquad \ds =\int_0^1 \p_t \big[(t-s)^{1-\varepsilon}\big](1-z)^{-\varepsilon} g( (t-s)z)dz\\[0.125in]
\ds \qquad \ds  + \int_0^1 (t-s)^{1-\varepsilon}(1-z)^{-\varepsilon} \p_t \big[g( (t-s)z)\big]dz : = K_1 + K_2.
\ds
  \end{array}
\end{equation}

We combine  the estimate  for $g$ in Lemma \ref{Sol:Lem0}
  as well as  the definition of the Beta function $B(\cdot, \cdot)$ \cite{Pod} to bound $K_1$ in \eqref{Lem1:e3} by
\begin{equation}\begin{array}{l}\label{Lem1:I1}
\ds |K_1| \le Q \int_0^1 (t-s)^{-\varepsilon}(1-z)^{-\varepsilon} |g( (t-s)z) |dz\\[0.1in]
\ds \qquad\le Q \int_0^1 (t-s)^{-\varepsilon}(1-z)^{-\varepsilon}     \big|\ln[(t-s)z]\big|dz\\[0.1in]
\ds  \qquad\le Q \int_0^1 (t-s)^{-\varepsilon-\alpha^*}(1-z)^{-\varepsilon}  z^{-\alpha^*} dz\\[0.1in]
\ds \qquad\le  Q (t-s)^{-\varepsilon-\alpha^*} B(1-\varepsilon,1-\alpha^* ) \le Q(t-s)^{-\varepsilon-\alpha^*}.
  \end{array}
\end{equation}

We then combine  the second estimate in \eqref{Lem0:est:e0} to similarly bound  $K_2$ in \eqref{Lem1:e3}
\begin{equation}\begin{array}{l}\label{Lem1:I2}
\ds |K_2| \le Q \int_0^1 (t-s)^{1-\varepsilon}(1-z)^{-\varepsilon} \p_y g( y)\big|_{y= (t-s)z} \f{d y}{dt} dz\\[0.1in]
\ds \qquad \le Q \int_0^1 (t-s)^{1-\varepsilon}(1-z)^{-\varepsilon} [(t-s)z]^{-1} z \,dz \\[0.1in]
\ds \qquad  \le Q \int_0^1 (t-s)^{-\varepsilon}(1-z)^{-\varepsilon}  dz \le  Q (t-s)^{-\varepsilon},
  \end{array}
\end{equation}
which, combined with  \eqref{Lem1:e2}, \eqref{Lem1:e3} and \eqref{Lem1:I1}, yields that
\begin{equation}\label{Lem1:e4}\begin{array}{l}
\ds\big| {}^{R}\p_t^{\varepsilon} \big (B(t) v(\bm x,t) \big)\big|
\ds \le Q \int_0^t \big|\Delta v(\bm x,s)\big|  (t-s)^{-\varepsilon-\alpha^*} ds.
\end{array}\end{equation}

We apply  \eqref{Lem1:e1} and  \eqref{Lem1:e4} to bound \eqref{Lem1:Inv} by
\begin{equation}\label{Lem1:e5}\begin{array}{l}
\ds \bigg \| \int_0^t E^\prime(t-s)  (B(s) v(\bm x, s)) ds \bigg\|_{L^2} \\[0.15in]
\ds \quad \le Q \int_0^t \f{1}{(t-s)^{1-\varepsilon}} \int_0^s (s-\iota)^{-\varepsilon-\alpha^*} \|\Delta v(\cdot,\iota) \|_{L^2} d\iota ds \\ [0.1in]
\ds \quad = Q \int_0^t \|\Delta v(\cdot,\iota) \|_{L^2}  \int_\iota^t \f{ds d\iota}{(t-s)^{1-\varepsilon}(s-\iota)^{\varepsilon+\alpha^*}} \le Q  \int_0^t \f{\|\Delta v(\cdot,\iota)\|_{L^2}d\iota}{(t-\iota)^{\alpha^*}}.
\end{array}\end{equation}

We multiply \eqref{Lem1:e5} by $e^{-\gamma t}$, take the $L^p(0,T)$ norm in time, and apply Young's convolution inequality and the last inequality in \eqref{Model:e5} to obtain
\begin{equation*}\begin{array}{l}
\ds \bigg \| e^{-\gamma t}\int_0^t   E^\prime(t-s) B(s) v ds \bigg \|_{L^p(L^2)}  \le Q \bigg \| \int_0^t \f{e^{-\gamma (t-s)} }{(t-s)^{\alpha^*}} e^{-\gamma s} \|\Delta v(\cdot, s)\|_{L^2} ds \bigg\|_{L^p(0, T)}\\[0.1in]
\ds \qquad
\ds \qquad \le Q \int_0^T e^{-\gamma t} t^{-\alpha^*} dt ~ \| v(\cdot, t) \|_{\mathcal{X}_\gamma}
\le Q\gamma^{\alpha^*-1} \| v \|_{\mathcal{X}_\gamma},
\end{array}\end{equation*}
we thus complete the proof of the lemma.
\end{proof}
\subsection{Well-posedness of model (\ref{VtFDEs})}
We first prove the well-posedness of the equivalent model (\ref{VtFDEs1}) in the following theorem.
\begin{thm}\label{thm:Stab}
Suppose that  Assumption A holds, the problem \eqref{VtFDEs1} has a unique solution $u\in W^{1,p}(L^2)\cap L^p(\check{H}^2) $ for $1 < p < \infty $, and
\begin{equation*}\label{StabEstimate}
\|u\|_{W^{1,p}(L^2)} +\|u\|_{L^p(\check{H}^2)}\le Q \big (\|f\|_{L^p(L^2)} + \| u_0\|_{\check H^2} \big ), \quad Q=Q(\alpha^*,T,p).
\end{equation*}
\end{thm}

\begin{proof}
We prove the theorem in two steps.
\paragraph{Step 1: Well-posedness and contraction property of $\mathcal M$}
We define a map $\mathcal M$ by $w := \mathcal M v$ for any $v \in \mathcal X_\gamma$
\begin{equation} \label{mapM}\begin{array}{c}
\dt w(\bm x,t) - \Delta w(\bm x,t) = f(\bm x,t) + B(t) v(\bm x,t), \quad (\bm x,t)\in\Omega\times(0,T]; \\[0.075in]
w(\bm x,t)=0, \quad (\bm x,t) \in \partial\Omega\times(0,T];  \qquad w(\bm x,0)=0, \quad \bm x \in \Omega.
\end{array}\end{equation}
We employ that fact that $ B(t) v = g( t)* \Delta v$, the estimates for $g( t)$ in Lemma \ref{Sol:Lem0} and  Young's convolution inequality \cite{AdaFou} to bound
\begin{equation}\label{mapM:e2}\begin{array}{rl}
\ds\hspace{-0.1in} \big \| e^{-\gamma t} B(t) v \big \|_{L^p(L^2)}
&\ds = \bigg \|e^{-\gamma t}\int_0^t g( t-s) \Delta v(\bm x,s) d s \bigg\|_{L^p(L^2)}\\[0.15in]
&\ds\leq Q \bigg \| \int_0^t e^{-\gamma (t-s)} |\ln( t-s) | e^{-\gamma s}|\Delta v(\bm x,s) | ds \bigg \|_{L^p(L^2)}\\[0.15in]
& \ds \leq Q \big \|\big( e^{-\gamma t}t^{-\alpha_*}\big)*\big(e^{-\gamma t} \big \|\Delta v \big \|_{L^2(\Omega)}\big)\big \|_{L^p(0,T)}\\[0.075in]
&\ds  \leq Q\gamma^{\alpha_*-1} \big \|e^{-\gamma t} \Delta v \big \|_{L^p(L^2)} =  Q\gamma^{\alpha_*-1} \|v\|_{\mathcal X_\gamma}.
\end{array}\end{equation}
Hence $B(t) v \in L^p(L^2)$. By Lemma \ref{lem:Lp} problem \eqref{mapM} has a unique solution $w \in \mathcal X_\gamma$ and the mapping $\mathcal M: \mathcal X_\gamma \rightarrow X_\gamma$ is well defined.

Let $w_1 = \mathcal Mv_1$ and $w_2 = \mathcal Mv_2$ for $v_1,v_2 \in \mathcal X_\gamma$, then $w = w_1 - w_2 \in \mathcal X_\gamma$ satisfies
\begin{equation*} \label{map_1}
\ds \dt w(\bm x,t) - \Delta w(\bm x,t)  = B(t) v(\bm x,t),~~v:=v_1-v_2,
\end{equation*}
equipped with the zero initial and boundary conditions. Then \eqref{HeatSoln} becomes
\begin{equation}\label{mapM:e3}
w(\bm x,t) = \int_0^{t} E(t-s) \big(B(s) v(\bm x,s)  \big)d s.
\end{equation}
We apply the operator $\Delta$ on both sides of  (\ref{mapM:e3}), multiply the resulting equation by $e^{-\gamma t}$ and then incorporate $E^{\prime}(t) = \Delta E(t)$ to get
\begin{equation}\begin{array}{l}\label{mapM:e4}
\ds e^{-\gamma t} \Delta w(\bm x,t) = e^{-\gamma t}\int_0^t \Delta E(t-s)\big(B(s) v(\bm x,s)  \big)ds\\[0.1in]
\ds \qquad  \qquad \qquad= e^{-\gamma t}\int_0^t  E^{\prime}(t-s) \big(B(s) v(\bm x,s)  \big)ds.
\end{array}
\end{equation}
Use Lemma \ref{Sol:Lem1} to bound the right-hand side term of \eqref{mapM:e4} by
\begin{equation*}\label{Thm10:e4}
\|  w\|_{\mathcal{X}_{\gamma}} \le Q \gamma^{\alpha^*-1} \| v \|_{\mathcal{X}_{\gamma}}.
\end{equation*}
Consequently,  choose $\gamma$ large enough such that $Q \gamma^{\alpha^*-1} \le Q^* < 1$ to ensure that the mapping $\mathcal M: \mathcal X_\gamma \rightarrow \mathcal X_\gamma$ is a contraction.
 By the Banach fixed point theorem, $\mathcal M$ has a unique fixed point $w = \mathcal M w \in \mathcal X_\gamma$, that is, \eqref{VtFDEs} with $u_0=0$ has a unique solution in $\mathcal X$.

\paragraph{Step 2. Stability estimate of $u$} With the fixed point $w$, problem \eqref{mapM} becomes
\begin{equation*}
\begin{array}{c}
\dt w(\bm x,t) - \Delta w(\bm x,t) = f(\bm x,t) + B(t) w(\bm x,t), ~~(\bm x,t)\in\Omega\times(0,T]; \\[0.075in]
w(\bm x,t)=0, ~~(\bm x,t) \in \partial\Omega\times(0,T];  \quad w(\bm x,0)=0, ~~\bm x \in \Omega.
\end{array}\end{equation*}
We multiply $e^{-\gamma t}$ on both sides to rewrite the equation as
$$\p_t(e^{-\gamma t}w)-\Delta (e^{-\gamma t}w)=-\gamma e^{-\gamma t}w+e^{-\gamma t}f + e^{-\gamma t}B(t) w.$$
We apply (\ref{mapM:e2}) and Lemma \ref{lem:Lp} to conclude that for any $0<\bar t\leq T$
\begin{equation}\label{Stab:e1}\begin{array}{l}
\hspace{-0.1in} \| e^{-\gamma t}w \|_{W^{1,p}(0,\bar t;L^2)} +\| e^{-\gamma t}w \|_{L^p(0,\bar t;\check H^2)}  \\[0.075in]
\ds \quad \le Q \big\|-\gamma e^{-\gamma t}w+e^{-\gamma t}f + e^{-\gamma t}B(t) w\big\|_{L^p(0,\bar t;L^2)}\\[0.075in]
\ds \quad \leq Q \big ( \| f \|_{L^p(0,\bar t;L^2)} +\gamma \|e^{-\gamma t}w\|_{L^p(0,\bar t;L^2)}
+ \gamma^{\alpha_*-1}\big \|e^{-\gamma t} w \big \|_{L^p(0,\bar t;\check H^2)} \big ).
\end{array}\end{equation}
Choosing $\gamma$ in \eqref{Stab:e1} sufficiently large yields
\begin{equation}\label{Stab:e2}
\hspace{-0.1in}\big \|e^{-\gamma t}\dt w \big \|_{L^p(0,\bar t;L^2)}  +\| e^{-\gamma t}w \|_{L^p(0,\bar t;\check H^2)} \le Q\gamma \|e^{-\gamma t}w\|_{L^p(0,\bar t;L^2)}+Q \| f \|_{L^p(0,\bar t;L^2)},
\end{equation}
where $Q$ is independent of $\gamma$. We bound the first term on the right-hand side by H\"older's inequality as
$$\begin{array}{l}
\ds \big \|e^{-\gamma t}w \big \|_{L^p(0,\bar t;L^2)}^p=\int_0^{\bar t}\Big\| e^{-\gamma t}\int_0^t \p_s w(\cdot,s)ds\Big\|^p_{L^2}dt \\[0.1in]
\ds\quad \leq Q\int_0^{\bar t}\Big(\int_0^t \big\| e^{-\gamma s} \p_s w(\cdot,s)\big\|_{L^2}ds\Big)^p\d t 
\leq Q \int_0^{\bar t}\|e^{-\gamma s}\p_s w\|^p_{L^p(0,t;L^2)}d t.
\end{array}$$
We plug this estimate into (\ref{Stab:e2}) to find
\begin{equation*}\begin{array}{l}
\big \|e^{\gamma t}\dt w \big \|^p_{L^p(0,\bar t;L^2)}  +\| e^{-\gamma t}w \|_{L^p(0,\bar t;\check H^2)}^p \\[0.1in]
\ds \qquad  \le Q \gamma^p \int_0^{\bar t}\|e^{-\gamma s}\p_s w\|^p_{L^p(0,s;L^2)}d s+Q \| f \|^p_{L^p(0,T;L^2)}, ~~0<\bar t \leq T.
\end{array}
\end{equation*}
Application of Gronwall's inequality yields
\begin{equation}\label{Stab:e3}
\big \|e^{-\gamma t}\dt w \big \|_{L^p(0,\bar t;L^2)} + \| e^{-\gamma t}w \|_{L^p(0,\bar t;\check H^2)} \leq Q e^{Q\gamma^p T} \| f \|_{L^p(0,T;L^2)}, ~~0<\bar t \leq T.
\end{equation}

Let $u$ be the solution to problem \eqref{VtFDEs1}. Then $w(\bm x,t) := u(\bm x,t) - u_0(\bm x)$ satisfies
\begin{equation}\label{Stab:e4}\begin{array}{l}
\hspace{-0.05in} \dt w(\bm x,t) - \Delta w(\bm x,t) = f(\bm x,t) +B(t) w(\bm x,t) +\Delta u_0(\bm x)+ B(t) u_0, \\[0.05in]
\hspace{2in} ~(\bm x,t) \in \Omega\times(0,T]; \\[0.05in]
w(\bm x,t) = 0, ~~(\bm x,t)\in \partial\Omega\times(0,T]; \quad w(\bm x,0)=0, ~~\bm x \in \Omega.
\end{array}\end{equation}
Use the estimate \eqref{Stab:e3}  to problem \eqref{Stab:e4} and the estimates in Lemma \ref{Sol:Lem0} to conclude that for some positive $Q = Q(\gamma)$
\begin{equation*}\begin{array}{l}
\hspace{-0.1in} \ds \|u\|_{W^{1,p}(0,T;L^2)} \!+\! \|u\|_{L^{p}(0,T;\check H^2)} \le \| w \|_{W^{1,p}(0,T;L^2)} + \|w\|_{L^{p}(0,T;\check H^2)}+ \| u_0 \|_{\check H^2} \\[0.1in]
\ds \qquad \qquad \qquad \qquad \qquad \qquad  \le Q \big (\|f\|_{L^p(0,T;L^2)}\! + \!\| u_0\|_{\check H^2} \!+\! \|g * 1\|_{L^p(0,T)} \|\Delta u_0 \|_{L^2} \big )\\[0.1in]
\ds \qquad \qquad \qquad \qquad \qquad \qquad  \le Q \big (\|f\|_{L^p(0,T;L^2)} + \|u_0\|_{\check H^2} + \|g \|_{L^1(0,T)} \| \Delta u_0 \|_{L^2} \big )\\[0.1in]
\ds \qquad \qquad \qquad \qquad \qquad \qquad \le Q \big (\|f\|_{L^p(0,T;L^2)} + \| u_0\|_{\check H^2} ).
\end{array}
\end{equation*}
We thus complete the proof.
\end{proof}

\begin{cor}\label{Cor:Stab}
Suppose that  Assumption A holds, the multiscale diffusion model  \eqref{VtFDEs} has a unique solution $u\in W^{1,p}(L^2)\cap L^p(\check{H}^2) $ for $1 < p < \infty $, and
\begin{equation}\label{Cor:StabEstimate}
\|u\|_{W^{1,p}(L^2)} +\|u\|_{L^p(\check{H}^2)}\le Q \big (\|f\|_{L^p(L^2)} + \| u_0\|_{\check H^2} \big ), \quad Q=Q(\alpha^*,T,p).
\end{equation}
\end{cor}
\begin{proof}
By Theorem \ref{Model}, the problem \eqref{VtFDEs1} could be reformulated to the original problem \eqref{VtFDEs} and thus the regularity estimate \eqref{Cor:StabEstimate} follows from  that in  Theorem \ref{thm:Stab}.
 The uniqueness of the solution to the original problem \eqref{VtFDEs} in $W^{1,p}(L^2) \cap L^p(\check{H}^2)$ follows from that of the problem \eqref{VtFDEs1} by Theorem \ref{Model} and Theorem \ref{thm:Stab}, which completes the proof.
\end{proof}
\section{High-order regularity estimate}\label{S:Hreg}

 We use \eqref{HeatPDE} to express the solution $u$ to the problem \eqref{VtFDEs1} as
\begin{equation}\begin{array}{l}\label{VtSol:e1}
  \ds u(\bm x, t) = E(t ) u_0 + \int_0^t E(t-s) f(\bm x, s) ds + \int_0^t E(t-s) B(s) u (\bm x, s) ds\\[0.1in]
  \ds \qquad \quad := L_1 + L_2.
  \end{array}
\end{equation}

\begin{lem}\label{lem:Append}
Suppose  Assumption A holds and that  $\alpha\in W^{4,\infty}( 0, T)$.  There is a positive $Q=Q(\varepsilon,\alpha^*,\|\alpha\|_{W^{4,\infty}},T)$ such that for $0<t\leq T$ and $0 \le \varepsilon\ll 1$

Case 1: $ \alpha^\prime( 0) \neq 0 :$
\begin{equation}\label{lem:Append:e1}
\big |{}^R\p_t^\varepsilon \p_t \big(B(t) u\big)\big|
\leq Q \int_0^t \frac{ |\ln(t-\theta)|}{(t-\theta)^{\varepsilon}} \big |\p_\theta \Delta u(\bm x,\theta) \big |d\theta +\f{Q |\ln(t)|}{t^{\varepsilon}}|\Delta u_0|.
\end{equation}

Case 2: $\alpha^\prime( 0) = 0:$
\begin{equation}\label{lem:Append:e2}
\big |{}^R\p_t^\varepsilon \p_t \big(B(t) u\big)\big|
\leq Q \int_0^t \frac{\big |\p_\theta \Delta u(\bm x,\theta) \big |d\theta}{(t-\theta)^{\varepsilon}} +\f{Q |\Delta u_0|}{t^{\varepsilon}}.
\end{equation}
\end{lem}
\begin{proof}
The proof of this lemma is given in the Appendix.
\end{proof}

\begin{thm}\label{thm:utt}
Suppose Assumption A holds and that $\alpha\in W^{4,\infty}( 0, T)$,  $\Delta u_0$, $\Delta^2 u_0\in L^2$ and $f\in W^{1,p}(L^2)\cap L^p(\check H^{2+\sigma})$ for $0 < \sigma \ll 1$ and $ 1 \le p \le\infty$.
If  $ \alpha^\prime( 0) \neq 0$, then the following estimate holds for $1 \le p < \infty$
\begin{equation}\label{thm:utt:e1}
\begin{array}{l}
\ds \|u\|_{W^{2,p}(L^2)}+\|u\|_{L^\infty(\check H^2)}+\|u\|_{W^{1,p}(\check H^2)}\\[0.075in]
\ds\qquad\quad\leq Q\big( \|f\|_{W^{1,p}(L^2)} + \| f\|_{L^p(\check H^{2+\sigma})} +\|\Delta u_0\|+\|\Delta^2 u_0\|\big).
\end{array}
\end{equation}
Here  $Q=Q(p,\alpha^*,\|\alpha\|_{W^{4,\infty}},T)$.

Furthermore, if $ \alpha^\prime( 0)  = 0$, then the above estimate holds for $1 \le p \le \infty$.
\end{thm}

\begin{proof}
We first consider $ \alpha^\prime( 0)  \neq 0$.
We apply $E^\prime(t) = E(t) \Delta$ and \eqref{Et:Express} to directly evaluate $ \p_t^2 L_1$ in \eqref{VtSol:e1} by
\begin{equation}\label{P2L1}\begin{array}{rl}
\p_t^2L_1 & \hspace{-0.1in}\ds = E(t) \Delta^2 u_0 + \p_t f +\Delta f +\int_0^t\sum_{i=1}^\infty \lambda_i^2e^{-\lambda_i (t-s)}( f,\phi_i)\phi_i ds\\[0.1in]
\ds &  \hspace{-0.1in} \ds = E(t) \Delta^2 u_0 + \p_t f +\Delta f +\int_0^t E^\prime (t-s) \Delta f(\bm x, s) ds.
\end{array}\end{equation}
Use \eqref{E:est} to bound
\begin{equation}\label{wtt:e1}\begin{array}{rl}
\hspace{-0.2in}\ds \|\p_t^2L_1\|_{L^2(\Omega)} & \ds \hspace{-0.1in} \ds\le  \|E(t)\|_{L^2 \rightarrow L^2} \|\Delta^2 u_0\|_{L^2} + \|\p_t f\|_{L^2} + \|\Delta f\|_{L^2}  \\[0.05in]
\ds \qquad  & \ds \hspace{-0.1in} \qquad  + \int_0^t \big\| E^\prime (t-s) \Delta f(\bm x, s)\big\|_{L^2} ds\\[0.1in]
& \hspace{-0.1in} \ds \le \|\Delta^2 u_0\|_{L^2} + \|\p_t f\|_{L^2} + \|\Delta f\|_{L^2} + \int_0^t \big\| E (t-s) \Delta f(\bm x, s)\big\|_{\check H^2} ds \\[0.1in]
\ds & \hspace{-0.1in} \ds \le \|\Delta^2 u_0\|_{L^2} \!+\! \|\p_t f\|_{L^2} \!+\! \|\Delta f\|_{L^2} \!+ \! \int_0^t (t-s)^{-\f{2-\sigma}{2}} \|\Delta f(\cdot, s)\|_{\check H^\sigma}ds
\end{array}\end{equation}
for $0 < \sigma \ll 1$.
Take $\|\cdot\|_{L^p(0, T)}$ norm on both sides of the above inequality and apply Young's inequality to bound
\begin{equation}\label{wtt:e2}\begin{array}{rl}
\hspace{-0.21in}\ds \|\p_t^2 L_1\|_{L^p(L^2)}  & \hspace{-0.125in} \ds \!\le\! \|\Delta^2 u_0\|_{L^2}  \!+\! \|f\|_{W^{1,p}(L^2)}   \!+\! \|\Delta f\|_{L^p(L^2)} \!+\!  \big\|t^{-\f{2-\sigma}{2}} \!*\! \| f\|_{\check H^{2+\sigma}} \big\|_{L^p(0, T)}\\[0.1in]
\ds & \hspace{-0.125in} \ds \!\le\! \|\Delta^2 u_0\|_{L^2} + \|f\|_{W^{1,p}(L^2)} + \| f\|_{L^p(\check H^{2+\sigma})}.
\end{array}\end{equation}

We utilize the commutativity of convolution operator to obtain
$$\begin{array}{l}
\ds \p_t \int_0^{t} E^\prime(t-s)\,  \big( B(s) u (\bm x, s)  \big)ds
= \p_t \int_0^{t} E^\prime(s)  \big ( B(y) u (\bm x, y)  \big ) \big |_{y=t-s}\,ds\\[0.1in]
\ds \quad = \int_0^{t} E^\prime(s)  \p_t\big(B(y) u (\bm x, y)\big) \big |_{y=t-s}\big) ds
= -\int_0^{t} E^\prime(t-s) \p_s\big(B(s) u (\bm x, s) \big) ds.
\end{array}$$
Differentiate $L_2$ in \eqref{VtSol:e1} with respect to $t$ twice and use above relation to obtain
\begin{equation}\begin{array}{l}\label{wtt:e3}
\ds \p_t L_2 = \int_0^{t} E^\prime(t-s)\, \big( B(s) u (\bm x, s)  \big)ds +B(t) u(\bm x,t),\\[0.125in]
\ds \p_t^2 L_2 = -\int_0^{t} E^\prime(t-s) \p_s\big(B(s) u (\bm x, s) \big) ds + \p_t\big(B(t) u(\bm x,t)\big).
\end{array}
\end{equation}
We use the estimate \eqref{lem:Append:e1} with $\varepsilon=0$, that is,
 \begin{equation}\begin{array}{l}\label{wtt:E3}
   \ds \big|\p_t(B(t)u)\big| \le Q \int_0^t \big|\ln(t-\theta)\big| \big|\p_\theta \Delta u\big| d \theta + Q |\ln(t)| |\Delta u_0|
   \end{array}
 \end{equation}
 combined with $\big|\ln(t) \big| \le Q t^{-\varepsilon}$ for $0 < \varepsilon \ll 1$ to directly bound the second term on the right hand side of the second equality in \eqref{wtt:e3}   and follow the preceding procedures in Lemma \ref{Sol:Lem1} to reformulate
$$\begin{array}{l}
\ds \int_0^{t} E^\prime(t-s)  \p_s\big (B(s) u (\bm x, s)\big) ds \\[0.05in]
\ds \quad = \int_0^t\bigg[\frac{1}{2\pi\rm i} \int_{\Gamma_\theta}z^{1-\varepsilon} (z-\Delta)^{-1} e^{z(t-s)} d z\bigg]\big({}^R\partial_s^\varepsilon \p_s(B(s) u (\bm x, s)\big)ds. \end{array}$$
We  invoke  (\ref{GammaEstimate}) to bound the integral in the square brackets, and then  incorporate   the estimate \eqref{lem:Append:e1} to bound the first term on the right-hand side of $\p_t^2 L_2$  in \eqref{wtt:e3}
\begin{equation}\label{wtt:e4}\begin{array}{l}
\hspace{-0.15in}\ds \bigg \| \int_0^{t} E^\prime(t-s)  \p_s\big (B(s) u (\bm x, s)\big) ds\bigg\|_{L^2}\leq Q\int_0^t \f{\big \|{}^R\partial_s^\varepsilon \p_s(B(s) u(\bm x,s)) \big \|_{L^2} ds}{(t-s)^{1-\varepsilon}}\\[0.1in]
\hspace{-0.1in} \ds\leq Q\int_0^t \f1{(t-s)^{1-\varepsilon}} \Big(\int_0^s \frac{|\ln(s-\theta)|}{(s-\theta)^{\varepsilon}} \|\partial_{\theta} \Delta u(\cdot,\theta)\|_{L^2} d\theta + \|\Delta u_0\| s^{-\varepsilon} |\ln(s)| \Big)d s\\[0.15in]
\hspace{-0.1in}\ds\leq Q\int_0^t \f1{(t-s)^{1-\varepsilon}} \Big(\int_0^s \frac{\|\partial_{\theta} \Delta u(\cdot,\theta)\| _{L^2} }{(s-\theta)^{2\varepsilon}} d\theta + \|\Delta u_0\| s^{-2\varepsilon}\Big) d s\\[0.15in]
\hspace{-0.1in}\ds\leq Q\int_0^t\frac{\|\partial_{\theta} \Delta u(\cdot,\theta)\| _{L^2}}{(t-\theta)^{\varepsilon}}d\theta+ Q \|\Delta u_0\|t^{-\varepsilon}.
\end{array}\end{equation}

Fix  $1 \le p < \infty$, and we can then choose $ 0 < \varepsilon \ll 1$ such that $0 < \varepsilon p <1 $.  we multiply $\p_t^2 L_2$ in \eqref{wtt:e3} by $e^{-\gamma t}$   and take the $\|\cdot\|_{L^p(0,T)}$ on both sides of the resulting equation and then invoke \eqref{Model:e5}, \eqref{wtt:E3}, \eqref{wtt:e4} and Young's convolution inequality to obtain
\begin{equation}\label{wtt:L2}\begin{array}{l}
\ds \big \|e^{-\gamma t} \p_t^2 L_2 \big \|_{L^p(L^2)}  \leq Q \big \|\big(e^{-\gamma t}t^{-\varepsilon}\big)*\big(e^{-\gamma t}\|\p_t \Delta u(\cdot,t)\| \big) \big \|_{L^p(0,T)} +Q\|\Delta u_0\| \\[0.1in]
\qquad \qquad \qquad \qquad\ds \leq Q\gamma^{\varepsilon-1}\|e^{-\gamma t}\p_t \Delta u\|_{L^p(L^2)}+Q \|\Delta u_0\|.
\end{array}\end{equation}
We differentiate \eqref{VtSol:e1} twice in time, apply the $\|\cdot\|_{L^p(0,T)}$ norm on both sides of the resulting equation multiplied by $e^{-\gamma t}$ and invoke \eqref{wtt:e2} and \eqref{wtt:L2} to obtain
\begin{equation}\label{wtt:e5}\begin{array}{l}
\hspace{-0.2in}\ds \big \|e^{-\gamma t} \p_t^2 u \big \|_{L^p(L^2)}  \ds \leq \big \|e^{-\gamma t}\p_t^2L_1 \big \|_{L^p(L^2)}
+\big \|e^{-\gamma t}\p_t^2L_2 \big \|_{L^p(L^2)} \\[0.075in]
\hspace{-0.2in} \ds\leq  Q\big( \|f\|_{W^{1,p}(L^2)} \!+\! \| f\|_{L^p(\check H^{2+\sigma})} \!+\!\|\Delta u_0\|\!+\!\|\Delta^2 u_0\|\big) \!+\! Q\gamma^{\varepsilon-1} \big \| e^{-\gamma t}\p_t \Delta u \big \|_{L^p(L^2)}.
\end{array}\end{equation}
We differentiate  \eqref{VtFDEs1} with respect to $t$,  apply the $\|\cdot\|_{L^p(0,T)}$ norm on both sides of the resulting equation multiplied by $e^{-\gamma t}$, and combine  the estimates \eqref{wtt:E3}, \eqref{wtt:e5} with Young's convolution inequality to bound
\begin{equation*}\begin{array}{rl}
\big \|e^{-\gamma t} \p_t \Delta u \big \|_{L^p(L^2)} & \hspace{-0.1in}\ds \leq \big \| e^{-\gamma t} \big(\p_t^2 u- \p_t f -\p_t(B(t)u) \big)\big \|_{L^p(L^2)}\\[0.1in]
& \hspace{-0.1in}\ds \leq  Q\big( \|f\|_{W^{1,p}(L^2)} + \| f\|_{L^p(\check H^{2+\sigma})} +\|\Delta u_0\|+\|\Delta^2 u_0\|\big)\\[0.1in]
& \hspace{-0.1in}\ds \quad + Q\gamma^{\varepsilon-1} \big \| e^{-\gamma t}\p_t \Delta u \big \|_{L^p(L^2)}.
\end{array}\end{equation*}
We set $\gamma$ large enough to hide the last term on the right-hand side to get
\begin{equation}\label{wtt:e6}
\hspace{-0.1in}\ds \big \| \p_t \Delta u\big \|_{L^p(L^2)}\!=\! \big \| \p_t  u\big \|_{L^p(\check H^2)}\leq   Q\big( \|f\|_{W^{1,p}(L^2)} \!+\! \| f\|_{L^p(\check H^{2+\sigma})} \!+\!\|\Delta u_0\|+\|\Delta^2 u_0\|\big),
\end{equation}
which combined with   \eqref{wtt:e5}, yields that
\begin{equation}\label{wtt:e7}\begin{array}{rl}
\big \|\p_t^2 u \big \|_{L^p(\check H^2)} \le Q\big( \|f\|_{W^{1,p}(L^2)} + \| f\|_{L^p(\check H^{2+\sigma})} +\|\Delta u_0\|+\|\Delta^2 u_0\|\big).
\end{array}\end{equation}
We finally employ the estimate \eqref{wtt:e6} and  $\|u\|_{L^\infty(\check H^2)} \le Q  \|   u \|_{W^{1,p}(\check H^2)}$ by the Sobolev embedding $W^{1,p}(0, T) \hookrightarrow  L^\infty(0, T)$ to arrive at the estimate for  $\|u\|_{L^\infty(\check H^2)}$ in \eqref{thm:utt:e1}.

If $ \alpha^\prime( 0) = 0$,  the estimates \eqref{wtt:E3} and \eqref{wtt:e4} could be augmented by the estimate \eqref{lem:Append:e2}  as follows
 \begin{equation}\label{wtt:e8}
 \begin{array}{l}
 \ds \big|\p_t(B(t)u)\big| \le Q \int_0^t \big|\p_\theta \Delta u\big| d \theta + Q |\Delta u_0|,\\[0.1in]
 \ds \ds \bigg \| \int_0^{t} E^\prime(t-s)  \p_s\big (B(s) u (\bm x, s)\big) ds\bigg\|_{L^2} \leq Q\int_0^t \|\partial_{\theta} \Delta u(\cdot,\theta)\| _{L^2} d\theta+ Q \|\Delta u_0\|,
 \end{array}
 \end{equation}
from which we observe that there is no singularity for the right hand side terms. We then follow the similar derivations to arrive at the estimates \eqref{thm:utt:e1} for $1 \le p \le \infty$ and thus complete the proof of the theorem.
\end{proof}

%

%
%

\section{Numerical computation}\label{S:Nume}
\subsection{Numerical discretization}
Partition $[0,T]$ by $t_n := n\tau$ for $\tau := T/N$ and $0\leq n\leq N$. Define a quasi-uniform partition of $\Omega$ with mesh diameter $h$ and let $S_h$ be the space of continuous and piecewise linear functions on $\Omega$ with respect to the partition. Let $I$ be the identity operator. The Ritz projection $\Pi_h:H^1_0(\Omega)\rightarrow S_h(\Omega)$ defined by
$\big(\nabla(g-\Pi_h g),\nabla \chi\big)=0$ for any $ \chi\in S_h$ has the approximation property \cite{Tho}
\begin{equation}\label{Ritz:e2}
\big \| (I - \Pi_h) g \big \| \leq Q h^2 \|g\|_{H^2}, \quad \forall g\in H^2(\Omega) \cap H_0^1(\Omega).
\end{equation}

Let $u_n:=u(\bm x,t_n)$ and $f_n:=f(\bm x,t_n)$. We discretize $\p_t u$ and $g* \Delta u$ with $g$ defined in \eqref{Lem0:e1} at $t=t_n$ for $1\leq n\leq N$ by
\begin{equation}\label{FEMdist}\begin{array}{rl}
\ds \p_t u \big|_{t=t_n} & \hspace{-0.1in}\ds = \delta_{\tau} u_n + E_n
:= \frac{u_n - u_{n-1}}{\tau} + \f1{\tau}\int_{t_{n-1}}^{t_n} \p_t^2 u(\bm x,t)(t-t_{n-1})dt, \\[0.15in]
 \ds g* \Delta u \big|_{t=t_n} &\hspace{-0.1in}\ds=  \sum_{k=1}^n \int_{t_{k-1}}^{t_k} g(t_n-s)\Delta u(\bm x,s) ds \\[0.175in]
&\hspace{-0.1in}\ds = \sum_{k=1}^n \int_{t_{k-1}}^{t_k} \! \f{(t_n-s)^{-\alpha(t_n-s)} G( t_n-s) }{\Gamma(1-\alpha(t_n-s))}\Delta u(\bm x,s) ds \\[0.175in]
&\hspace{-0.1in}\ds= I_{\tau}^{g( t_n)} \Delta u_n + J_n + \hat J_n.
\end{array}\end{equation}
Here $I_{\tau}^{g( t_n)} u_n$ is defined by
 \begin{equation}\label{dis:frac}\begin{array}{ll}
\hspace{-0.15in}&  \hspace{-0.125in}\ds \ds I_{\tau}^{g(t_n)}\Delta u_n := \sum_{k=1}^n \int_{t_{k-1}}^{t_k} \! \f{(t_n-s)^{-\alpha(t_n-t_k)} \breve G_{n,k}( s) }{\Gamma(1-\alpha(t_n-t_k))}\Delta u(\bm x,t_k) ds  = \sum_{k=1}^n  b_{n,k} \Delta u_k,\\[0.175in]
\hspace{-0.15in}& \hspace{-0.125in} \ds \ds \ds  \breve G_{n,k}( s)  := - \alpha^\prime( t_n-t_k) \ln(t_n-s) +R_{n,k} ,\\[0.175in]
\hspace{-0.15in}& \hspace{-0.125in} \ds R_{n,k}:= -\f{\alpha(t_n-t_k)}{t_n-t_k}   \ds + \f{\Gamma^\prime(1-\alpha(t_n-t_k))\alpha^\prime( t_n-t_k)}{\Gamma(1-\alpha( t_n-t_k))},\\[0.175in]
\hspace{-0.15in}& \hspace{-0.125in} \ds b_{n,k}  \ds := \int_{t_{k-1}}^{t_k} \! \f{(t_n-s)^{-\alpha(t_n-t_k)} \breve G_{n,k}( s) }{\Gamma(1-\alpha(t_n-t_k))} ds  = \f{-\alpha^\prime( t_n-t_k) b_{n,k}^1+{R_{n,k}} b_{n,k}^2}{\Gamma(1-\alpha(t_n-t_k))}.
\end{array}
\end{equation}
By {\it Assumption A}, we find that
$$
\begin{array}{ll}
  \ds \lim\limits_{t\rightarrow0} \frac{\alpha(t)}{t} = \alpha'(0), \quad \frac{\Gamma'(\kappa)}{\Gamma(\kappa)} = - \gamma_{e} + \int_{0}^{1}\frac{1-t^{\kappa-1}}{1-t}dt, \quad  0< \kappa \leq 1
\end{array}
$$
with $\frac{\Gamma'(1)}{\Gamma(1)}=-\gamma_{e}$ and $\gamma_{e}\approx 0.577$ denoting the Euler's constant, which, combined with {\it Assumption A}, ensures the boundedness of $R_{n,k}$ in  \eqref{dis:frac}. In addition,  $ b_{n,k}^1 $ and $ b_{n,k}^2$ in \eqref{dis:frac} are expressed as
$$
\begin{array}{ll}
\hspace{-0.15in}& \hspace{-0.125in}\ds b_{n,k}^1 : = \int_{t_{k-1}}^{t_k}  \f{\ln(t_n-s)}{(t_n-s)^{\alpha(t_n-t_{k})}}  ds  = \f{(t_n-t_{k-1})^{1-\alpha(t_n-t_k)}}{1-\alpha(t_n-t_k)}\Big(\ln(t_n-t_{k-1})\\[0.15in]
\hspace{-0.15in}& \hspace{-0.125in}\ds \quad  -\big(1-\alpha(t_n-t_{k})\big)^{-1}\Big)   - \f{(t_n-t_{k})^{1-\alpha(t_n-t_k)}}{1-\alpha(t_n-t_k)}\Big(\ln(t_n-t_{k}) - \big(1-\alpha(t_n-t_{k})\big)^{-1}\Big), \\[0.15in]
\hspace{-0.15in}& \hspace{-0.125in}\ds b_{n,k}^2 : = \int_{t_{k-1}}^{t_k}  (t_n-s)^{-\alpha(t_n-t_k)} ds = \f{(t_n-t_{k-1})^{1-\alpha(t_n-t_k)}-(t_n-t_{k})^{1-\alpha(t_n-t_k)}}{1-\alpha(t_n-t_k)},
\end{array}
$$
and $\hat J_n$ and $J_n$ in \eqref{dis:frac} are defined as follows
 \begin{equation}\label{dis:err}\begin{array}{rl}
\hspace{-0.1in} \ds \hat J_n &\ds := \sum_{k=1}^n \int_{t_{k-1}}^{t_k} \bigg [ g(t_n-s) - \f{(t_n-s)^{-\alpha(t_n-t_k)}}{\Gamma(1-\alpha(t_n-t_k))}   \breve G_{n,k}(s) \bigg ]\Delta u(\bm  x,s) ds,\\[0.15in]
\hspace{-0.1in} J_n &\ds :=\sum_{k=1}^n J_{n,k}:= \sum_{k=1}^n \int_{t_{k-1}}^{t_k} \frac{\big(\Delta u(\bm  x,s) - \Delta u_k \big) \breve G_{n,k}( s) ds}{\Gamma(1-\alpha(t_n-t_k))(t_n-s)^{\alpha(t_n-t_k)}} \\[0.175in]
&  \ds =  \sum_{k=1}^n \int_{t_{k-1}}^{t_k}\frac{\int_{t_k}^{s} \p_\theta \Delta u(\bm  x,\theta) d \theta  \breve G_{n,k}( s) }{\Gamma(1-\alpha(t_n-t_k))(t_n-s)^{\alpha(t_n-t_k)}}ds.
\end{array}\end{equation}
We plug these discretizations into (\ref{VtFDEs1}), and integrate the resulting equation multiplied by $\chi\in H_0^1(\Omega)$ on $\Omega$ to obtain its weak formulation  for any $\chi \in H_0^1(\Omega)$ and $n=1,2,\ldots,N$
\begin{equation}\label{WeakForm}\begin{array}{l}
\ds\hspace{-0.1in} (\delta_{\tau}  u_n,\chi)+ \big(\nabla u_n,\nabla \chi\big) + \big(I_{\tau}^{g(t_n)} \nabla u_n, \nabla\chi\big) =  (f_n,\chi)-\big( \hat J_n+J_n+E_n,\chi\big).
\end{array} \end{equation}
Drop the local truncation error terms in \eqref{WeakForm} to obtain a fully-discrete finite element scheme for (\ref{VtFDEs1}): find $U_n \in S_h$ with $U_0 := \Pi_h u_0$ such that for $n=1,\cdots,N$ and $\chi \in S_h$
\begin{equation}\label{FEM}
(\delta_{\tau}  U_n,\chi) +\big(\nabla U_n,\nabla \chi\big) +  \big(I_{\tau}^{g(t_n)} \nabla U_n,\nabla \chi\big)=  (f_n,\chi),~~\forall \chi \in S_h.
\end{equation}
\subsection{Error estimate}
We subtract equation \eqref{FEM} from equation \eqref{WeakForm} to obtain the following error equation  for $U_n-u_n$
\begin{equation}\label{ErrEq}\begin{array}{l}
\ds\hspace{-0.1in} (\delta_{\tau} (U-u)_n,\chi)+ \big(\nabla (U-u)_n,\nabla \chi\big) + \big(I_{\tau}^{g( t_n)} \nabla(U-u)_n,\nabla\chi\big) \\[0.1in]
\ds\qquad \qquad = \big( \hat J_n+J_n+E_n,\chi\big), \quad \forall \chi \in S_h.
\end{array} \end{equation}

Let $\Pi_h u$ be the Ritz projection of $u$ and $\eta: = \Pi_h u - u$ be bounded in \eqref{Ritz:e2}. We decompose the error $U_n - u_n = \xi_n + \eta_n$ and remain to bound $\xi_n := U_n - \Pi_h u_n \in S_h$.  We rewrite the error equation \eqref{ErrEq}  as follows
\begin{equation}\label{ErrEq:e1}\begin{array}{l}
\hspace{-0.1in} \ds (\delta_{\tau} \xi_n,\chi) +  \big(\nabla\xi_n,\nabla \chi\big) + \big(I_{\tau}^{g(t_n)} \nabla \xi_n,\nabla\chi\big) \ds= \big(\hat J_n+J_n + E_n - \delta_{\tau} \eta_n,\chi\big).
\end{array}\end{equation}

\begin{thm}\label{thm:Error} Suppose Assumption A holds and that $\alpha\in W^{4,\infty}( 0, T)$,  $\Delta u_0$, $\Delta^2 u_0\in L^2$ and $f\in W^{1,1}(L^2)\cap L^1(\check H^{2+\sigma})$ for $0 < \sigma \ll 1$. Then an optimal-order error estimate holds for the fully-discrete finite element scheme \eqref{FEM} for $ 0 < \tau \le \tau_0$ and for some $\tau_0>0$
\begin{equation}\label{thm:Err1}
\| U - u\|_{\hat L^\infty(L^2)}:=\max_{1\leq n\leq N}\|U_n-u_n\|\leq QM \big(\tau +h^2\big).
\end{equation}
Here $M :=  \|f\|_{W^{1,1}(L^2)} + \| f\|_{L^1(\check H^{2+\sigma})} +\|\Delta u_0\|+\|\Delta^2 u_0\|$, and $Q$ is independent of $f$, $u_0$, $h$, $\tau$, $N$.

\end{thm}
\begin{proof}
 For a fixed $\nu > 0$, define $\hat{\xi}_n := e^{- \nu t_n} \xi_n$ for $1 \le n \le N$. Multiply $e^{-\nu t_n}$ on both sides of \eqref{ErrEq:e1} and set $\chi = \hat{\xi}_n$  to obtain
 \begin{equation}\label{ErrEq:S1:e2}\begin{array}{l}
\ds \big (e^{- \nu t_n} \delta_{\tau} {\xi}_{n}, \hat{\xi}_{n} \big) + \big (\nabla \hat{\xi}_{n},\nabla \hat{\xi}_{n}\big)    +  \big( e^{-\nu t_n}I_{\tau}^{g(t_n)} \nabla \xi_n,\nabla \hat{\xi}_n\big) =  (G_n, \hat{\xi}_n)
\end{array}\end{equation}
with $G_n:= e^{-\nu t_n}(\hat J_n+J_n + E_n - \delta_{\tau} \eta_n )$.

We prove the theorem in two steps.

\paragraph{Step 1: Stability estimate of $\xi_n$} There is a constant $Q_0$ such that
\begin{equation}\label{Xi:e1}
\| \xi \|_{\hat L^\infty(L^2)} \leq Q_0 \| G \|_{\hat L^1(L^2)}, \qquad \| G \|_{\hat L^1(L^2)} := \tau \sum_{n=1}^N \|G_n\|,
\end{equation}
where $Q_0$ is independent of $u_0$, $f$, $\tau$, $N$, or $h$.

 We employ the relations
\begin{equation}\begin{array}{l}\label{ErrEq:S1:e3}
 \hspace{-0.1in} \ds e^{-\nu t_n} \delta_\tau \xi_n
  = \tau^{-1}\big[e^{-\nu t_n}\xi_n- e^{-\nu t_{n-1}}\xi_{n-1}  + e^{-\nu t_{n-1}}\xi_{n-1} - e^{-\nu t_n}\xi_{n-1}\big]\\[0.1in]
  \ds  \qquad \qquad = \delta_\tau \hat{\xi}_n +  \tau^{-1}(1-e^{-\nu \tau})\hat{\xi}_{n-1}=\tau^{-1} (\hat \xi_n - e^{-\nu \tau}\hat \xi_{n-1} ),
  \end{array}
  \end{equation}
 and
\begin{equation}\begin{array}{lr}\label{ErrEq:S1:e4}
\hspace{-0.1in}  \ds e^{-\nu t_n}   \big(  I_{\tau}^{g(t_n)} \nabla \xi_n, \nabla \hat{\xi}_{n}\big)  \ds= e^{-\nu t_n}  \sum_{k=1}^n  b_{n, k} ( \nabla \xi_k,\nabla\hat{\xi}_{n})\\[0.15in]
 \hspace{-0.1in}   \ds\ds= \sum_{k=1}^n \int_{t_k-1}^{t_k} \frac{ e^{-\nu (t_n-t_k)}   \breve G_{n,k}(s)   ds}{\Gamma(1-\alpha(t_n-t_k)) (t_n-s)^{\alpha(t_n-t_k)}} ( \nabla\hat{\xi}_{k}, \nabla\hat{\xi}_{n})  \\[0.15in]
 \hspace{-0.1in} \ds = : \sum_{k=1}^n  \hat{b}_{n,k} ( \nabla\hat{\xi}_{k}, \nabla\hat{\xi}_{n}),\\[0.175in]
\hspace{-0.1in}     \ds  \hat{b}_{n,k} \ds: =  \int_{t_k-1}^{t_k}  \frac{ e^{-\nu (t_n-t_k)} \breve G_{n,k}(s)   ds}{\Gamma(1-\alpha(t_n-t_k)) (t_n-s)^{\alpha(t_n-t_k)} } , \quad 1 \le k \le n \le N
\end{array}
\end{equation}
by \eqref{dis:frac} to reformulate  \eqref{ErrEq:S1:e2} multiplied by $\tau$ as
\begin{equation}\label{ErrEq:S1:e5}\begin{array}{l}
\hspace{-0.175in}\ds \big ( \hat{\xi}_{n}, \hat{\xi}_{n} \big) + \tau\big (\nabla \hat{\xi}_{n},\nabla \hat{\xi}_{n}\big)    +   \tau \sum_{k=1}^n  \hat{b}_{n, k} ( \nabla \hat{\xi}_k, \nabla \hat{\xi}_n) \!=\! e^{-\nu \tau}( \hat{\xi}_{n-1} , \hat{\xi}_{n}) + \tau (G_n, \hat{\xi}_n).
\end{array}\end{equation}
We then apply Cauchy's inequality and geometric-arithmetic inequality to cancel $\| \hat \xi_n \|^2/2$   on both sides and sum the resulting inequality multiplied by $2$ from $n=1$ to $n_*$ for $1 \le n_* \le N$   to obtain
\begin{equation}\label{ErrEq:S1:e6}\begin{array}{l}
\hspace{-0.175in} \ds \| \hat \xi_{n_*}\|^2  +2  \tau  \sum_{n=1}^{n_* } \|\nabla \hat \xi_n\|^2 \le   \tau  \sum_{n=1}^{n_* } \sum_{k=1}^n \big| \hat b_{n, k} \big|\|\nabla \hat \xi_k\|^2 + \tau  \sum_{n=1}^{n_* } \sum_{k=1}^n \big| \hat b_{n, k}\big|  \|  \nabla \hat \xi_n\|^2  \\[0.1in]
\ds \qquad \qquad  \qquad \qquad  \qquad   + 2 \tau \sum_{n=1}^{n_*}  \|G_n\| \|\hat \xi_n\|.
\end{array}
\end{equation}
To bound the first two terms on the right hand side of \eqref{ErrEq:S1:e6}, we firstly incorporate \eqref{ErrEq:S1:e4} to  derive the following estimate for $\hat b_{n,k}$ by applying the estimates in \eqref{Lem0:e3} to $\breve G_{n,k}(s)$ defined in \eqref{dis:frac}
\begin{equation}\label{ErrEq:S1:bnk}\begin{array}{l}
\hspace{-0.2in}\ds \big| \hat b_{n, k} \big| \ds \le  Q \int_{t_{k-1}}^{t_k} \frac{ e^{-\nu (t_n-t_k)}\big |\breve G_{n,k}(s)\big|  ds}{(t_n-s)^{\alpha^*} }  \le  Q\int_{t_{k-1}}^{t_k}  \frac{ e^{-\nu (t_n-t_k)} |\ln(t_n-s)|  ds}{(t_n-s)^{\alpha^*} } \\[0.15in]
 \ds \ds \le  Q\int_{t_{k-1}}^{t_k}  \frac{ e^{-\nu (t_n-t_k)} ds}{(t_n-s)^{(\alpha^*+1)/2} } \le   Q e^{\nu \tau} \int_{t_{k-1}}^{t_k} \frac{ e^{-\nu (t_n-s)} ds}{(t_n-s)^{(\alpha^*+1)/2} }, \,\, 1 \le k \le n \le N,
\end{array}
\end{equation}
by which we bound the first term on the right-hand side of \eqref{ErrEq:S1:e6} as follows
\begin{equation}\label{ErrEq:S1:sum1}\begin{array}{l}
\hspace{-0.1in}\ds \sum_{n=1}^{n_* } \sum_{k=1}^n \big| \hat b_{n, k} \big|\|\nabla \hat \xi_k\|^2 \le Q e^{\nu \tau} \sum_{n=1}^{n_* } \sum_{k=1}^n \int_{t_k-1}^{t_k} \frac{ e^{-\nu (t_n-s)} ds}{(t_n-s)^{(\alpha^*+1)/2} } \|\nabla \hat \xi_k\|^2\\[0.15in]
\ds \qquad \qquad \qquad \quad  = Q e^{\nu \tau}\sum_{n=1}^{n_* } \bigg[\sum_{k=1}^n  \int_{-\tau}^{0} \frac{ e^{-\nu (t_n-t_k-s)} ds}{(t_n-t_k-s)^{(\alpha^*+1)/2} } \|\nabla \hat \xi_k\|^2\bigg]\\[0.1in]
\ds \qquad \qquad \qquad \quad = Q e^{\nu \tau} \sum_{n=1}^{n_*} \bigg[\sum_{k=1}^n  \mathcal B_{n-k} \mathcal X_k\bigg]  =  Q e^{\nu \tau} \sum_{n=1}^{n_*} (\mathcal B * \mathcal X)_n\\[0.15in]
\ds \qquad \qquad \qquad \quad \le  Q e^{\nu \tau} \|\mathcal B * \mathcal X\|_{l_1} ,
\end{array}
\end{equation}
where $\mathcal B * \mathcal X$ represents the convolution of the sequences $\mathcal B : = \{ \mathcal B_n\}_{n=1}^N$ and $\mathcal X: = \{ \mathcal X_n\}_{n=1}^N$ defined as follows
\begin{equation}\label{ErrEq:S1:e7}\begin{array}{l}
\ds \mathcal B_n : = \int_{-\tau}^{0} \frac{ e^{-\nu (t_n-s)} ds}{(t_n-s)^{(\alpha^*+1)/2} }, \quad  \mathcal X_n : =  \|\nabla \hat \xi_n\|^2,  \quad
\ds \|\mathcal X\|_{l_1} : = \sum_{n=1}^N |\mathcal X_n|.
\end{array}
\end{equation}
We follow the discrete Young’s convolution formula \cite{AdaFou}, i.e.,  $\|\mathcal B * \mathcal X\|_{l_1} \le  \|\mathcal B\|_{l_1}  \|\mathcal X\|_{l_1}$, and
\begin{equation}\label{ErrEq:S1:B}\begin{array}{l}
\ds \|\mathcal  B\|_{l_1} =  \sum_{n=1}^N  \int_{-\tau}^0 \frac{e^{-\nu (t_n-s)} ds}{
(t_{n}-s)^{(\alpha^*+1)/2}} = \sum_{n=1}^N  \int_{t_{n}}^{t_{n+1}} e^{-\nu s} s^{-(\alpha^*+1)/2}ds\\[0.15in]
\ds \qquad \qquad \qquad \quad   \le \int_0^{\infty}  e^{-\nu s} s^{-(\alpha^*+1)/2}ds  \le
 Q \nu^{(\alpha^*-1)/2},
\end{array}
\end{equation}
by \eqref{Model:e5} to further bound the left-hand side term in  \eqref{ErrEq:S1:sum1} by
\begin{equation}\begin{array}{l}\label{ErrEq:Young}
  \ds \ds \sum_{n=1}^{n_* } \sum_{k=1}^n \big| \hat b_{n, k} \big|\|\nabla \hat \xi_k\|^2 \le    Q e^{\nu \tau} \|\mathcal B\|_{l_1}  \|\mathcal X\|_{l_1} \le    Q_1 e^{\nu \tau} \nu^{(\alpha^*-1)/2} \sum_{n=1}^{n_*} \|\nabla \hat \xi_n\|^2.
  \end{array}
\end{equation}
 We invoke \eqref{Model:e5} and \eqref{ErrEq:S1:bnk}  to bound the second term on the right hand side of \eqref{ErrEq:S1:e6} in a similar manner
\begin{equation}\label{ErrEq:S1:sum2}\begin{array}{l}
\hspace{-0.1in}\ds \sum_{n=1}^{n_* } \sum_{k=1}^n \big| \hat b_{n, k} \big|\|\nabla \hat \xi_n\|^2 = \sum_{n=1}^{n_* } \|\nabla \hat \xi_n\|^2 \sum_{k=1}^n \big| \hat b_{n, k} \big| \\[0.125in]
\ds \qquad \qquad \qquad \qquad  \le Q e^{\nu \tau} \sum_{n=1}^{n_* } \|\nabla \hat \xi_n\|^2 \sum_{k=1}^n  \int_{t_k-1}^{t_k} \frac{ e^{-\nu (t_n-s)} ds}{(t_n-s)^{(\alpha^*+1)/2} }\\[0.125in]
\ds \qquad \quad \qquad \qquad \quad    \le  Q_2 e^{\nu \tau} \nu^{(\alpha^*-1)/2}\sum_{n=1}^{n_*} \|\nabla \hat \xi_n\|^2.
\end{array}
\end{equation}
We incorporate the estimates \eqref{ErrEq:Young}--\eqref{ErrEq:S1:sum2} in \eqref{ErrEq:S1:e6} to find
\begin{equation}\label{ErrEq:S1:e8}\begin{array}{l}
\hspace{-0.175in} \ds \| \hat \xi_{n_*}\|^2  +2  \tau  \sum_{n=1}^{n_* } \|\nabla \hat \xi_n\|^2 \le  Q_3  e^{\nu \tau} \nu^{(\alpha^*-1)/2}\tau \sum_{n=1}^{n_*} \|\nabla \hat \xi_n\|^2 + 2 \tau \sum_{n=1}^{n_*}  \|G_n\| \|\hat \xi_n\|.
\end{array}
\end{equation}
We choose $\tau>0$ sufficiently small such that $Q_3 e^{\nu \tau} \nu^{(\alpha^*-1)/2} <2$ to cancel the like terms on both sides of \eqref{ErrEq:S1:e8} to conclude that for $1 \le n_* \le N$,
$$
 \ds \| \hat \xi_{n_*}\|^2  + Q  \tau  \sum_{n=1}^{n_* } \|\nabla \hat \xi_n\|^2 \le 2 \tau \sum_{n=1}^{n_*}  \|G_n\| \|\hat \xi_n\|.
$$
Drop the second term on the left-hand side of the above inequality to arrive at
\begin{equation}\label{ErrEq:S1:e9}
  \ds \| \hat \xi_{n_*}\|^2   \le  2 \tau \sum_{n=1}^{n_*}  \|G_n\| \|\hat \xi_n\|, \quad 1 \le n_* \le N.
\end{equation}
Let $\|\hat \xi_{m_*}\| := \max_{1 \le m \le N}  \|\hat \xi_m\|$ (assumed positive without loss of generality). Set $m = m_*$ in \eqref{ErrEq:S1:e9} and divide the resulting inequality by $\|\hat \xi_{m_*}\|$ and incorporate $\|\xi_n\| \le e^{\nu T}\|\hat\xi_n\|$ to arrive at \eqref{Xi:e1}.

\paragraph{Step 2: Estimates for $E$, $J$, $\hat J$ defined in \eqref{FEMdist}--\eqref{dis:err} and $\delta_\tau \eta$}

\begin{equation*}\begin{array}{rl}
\ds \| E\|_{\hat L^1(L^2)}  \ds \leq Q \sum_{n=1}^N\int_{t_{n-1}}^{t_n} \|\p_t^2 u\| (t-t_{n-1})dt\leq Q \tau \sum_{n=1}^N\int_{t_{n-1}}^{t_n}\|\p_t^2 u\| dt \leq QM \tau.
\end{array}\end{equation*}
We  follow a similar estimate as \eqref{ErrEq:S1:bnk} to bound $|\breve G_{n,k}(s)| \le Q|\ln(t_n-s)|$   and thus to bound $J$ by
\begin{equation*}\begin{array}{rl}
\ds \|J\|_{\hat L^1(L^2)} & \ds\leq Q \tau \sum_{n=1}^N \sum_{k=1}^n \int_{t_{k-1}}^{t_k} (t_n-s)^{-\alpha^*} \big|\ln(t_n-s)\big| \int_{t_{k-1}}^{t_k}\|\p_{\theta} u(\cdot,\theta)\|_{\check H^2} d\theta ds\\[0.15in]
&	\ds  =Q \tau \sum_{k=1}^N \int_{t_{k-1}}^{t_k}\|\p_{\theta} u(\cdot,\theta)\|_{\check H^2} d\theta \sum_{n=k}^N \int_{t_{k-1}}^{t_k} (t_n-s)^{-(\alpha^*+1)/2}  ds\\[0.15in]
&	 \ds\leq Q \tau \sum_{k=1}^N \int_{t_{k-1}}^{t_k} \|\p_{\theta} u(\cdot,\theta)\|_{\check H^2} d\theta\le Q \tau \|u\|_{W^{1,1}(\check H^2)}\leq QM \tau.
\end{array}\end{equation*}
We invoke \eqref{Lem0:e1} to split $\hat J_{n}$ in \eqref{dis:err} as
\begin{equation}\label{J}\begin{array}{rl}
\hspace{-0.175in} \ds \hat J_{n}
&\hspace{-0.125in}\ds = \sum_{k=1}^n \int_{t_{k-1}}^{t_k} \bigg [ \f{(t_n-s)^{-\alpha(t_n-s)}}{\Gamma(1-\alpha(t_n-s))}  - \f{(t_n-s)^{-\alpha(t_n-t_k)}}{\Gamma(1-\alpha(t_n-t_k))}    \bigg ]G(t_n-s) \Delta u(\bm  x,s) ds\\[0.15in]
\hspace{-0.175in} \ds  &\hspace{-0.125in} \ds  +  \sum_{k=1}^n \int_{t_{k-1}}^{t_k} \f{(t_n-s)^{-\alpha(t_n-t_k)}}{\Gamma(1-\alpha(t_n-t_k))}   \Big[G(t_n-s)-\breve G_{n,k}(s)\Big] \Delta u(\bm  x,s) ds =: \hat J_{n}^1 + \hat J_{n}^2.
\end{array}\end{equation}
We invoke the estimate blow \eqref{Lem0:e2} to bound the first term on the right-hand side by
\begin{equation*}\begin{array}{l}\label{J1}
\ds \big\|\hat J^1\big\|_{\hat L^1(L^2)}\le Q \tau \sum_{n=1}^N \|u\|_{L^\infty(\check H^2)}   \sum_{k=1}^n \int_{t_{k-1}}^{t_k} \bigg|\f{(t_n-s)^{-\alpha(t_n-s)}}{\Gamma(1-\alpha(t_n-s))}    - \f{(t_n-s)^{-\alpha(t_n-t_k)}}{\Gamma(1-\alpha(t_n-t_k))}\bigg| \\[0.175in]
\ds \qquad \qquad \qquad  \qquad \qquad \qquad  \qquad \qquad \qquad \qquad \times |\ln(t_n-s)| ds \\[0.1in]
\ds \qquad  \le  \ds Q M  \sum_{k=1}^n\int_{t_{k-1}}^{t_k} \big|\ln(t_n-s)\big| \int_{s}^{t_k} \bigg|\p_z\bigg(\f{(t_n-s)^{-\alpha(t_n-z)}}{\Gamma(1-\alpha(t_n-z))}\bigg)\bigg| dz  ds\\[0.15in]
\ds \qquad \le \ds Q M  \tau  \sum_{k=1}^n \int_{t_{k-1}}^{t_k} (t_n-s)^{-(\alpha^*+1)/2}ds \le QM \tau.
\end{array}
\end{equation*}
We then combine \eqref{Lem0:e1}, \eqref{dis:frac} with the estimates \eqref{Lem0:e3}--\eqref{Ptg} to bound
\begin{equation*}\begin{array}{l}\label{J2}
\ds \big\|\hat J^2\big\|_{\hat L^1(L^2)}\le Q \tau \sum_{n=1}^N \|u\|_{L^\infty(\check H^2)} \sum_{k=1}^n  \int_{t_{k-1}}^{t_k} (t_n-s)^{-\alpha^*}   \big|G(t_n-s)-\breve G_{n,k}(s)\big|  ds\\[0.15in]
\ds \le Q M \sum_{k=1}^n  \int_{t_{k-1}}^{t_k} (t_n-s)^{-\alpha^*}  \bigg[\big|\ln(t_n-s)\big| \int_{s}^{t_k}  \big| \p_z \big(\alpha^{\prime}(t_n-z) \big)\big|dz   \\[0.15in]
\ds \qquad \qquad  + \int_s^{t_k} \bigg|\p_z \bigg(-\f{\alpha(t_n-z)}{t_n-z} + \f{\Gamma^\prime(1-\alpha(t_n-z)) \alpha^{\prime}(t_n-z)}{\Gamma(1-\alpha(t_n-z))}\bigg)\bigg|dz \bigg] ds\\[0.15in]
\ds \le Q M \tau \sum_{k=1}^n   \int_{t_{k-1}}^{t_k} (t_n-s)^{-(\alpha^*+1)/2}ds  \le QM \tau.
\end{array}
\end{equation*}
We invoke the above two estimates in \eqref{J} to bound $\big\|\hat J\big\|_{\hat L^1(L^2)}\le QM \tau$.
We then use \eqref{Ritz:e2} to bound  $\delta_{\tau} \eta$ by
\begin{equation*}\begin{array}{l}
\ds \big \| \delta_{\tau} \eta \big \|_{\hat L^1(L^2)} = \sum_{n=1}^N \bigg \| \int_{t_{n-1}}^{t_n} (I-\Pi_h)\p_t u dt \bigg \| \leq Qh^2 \| u\|_{W^{1,1}(H^2)} \leq QM h^2.
\end{array}\end{equation*}
We incorporate these local error estimates into \eqref{Xi:e1}  to obtain
\begin{equation*}\label{Err:e1}\begin{array}{l}
\hspace{-0.15in} \ds \| \xi \|_{\hat L^\infty(L^2)}  \le Q M \big ( \|E\|_{\hat L^1(L^2)} +\|J\|_{\hat L^1(L^2)}+\| \hat J\|_{\hat L^1(L^2)} + \| \delta_{\tau} \eta \|_{\hat L^1(L^2)}  \big)   \le Q M (\tau + h^2),
\end{array}\end{equation*}
which, together  with \eqref{Ritz:e2}, yields \eqref{thm:Err1}.
\end{proof}

\section{Numerical simulation}\label{S:Experiment}
In this section, we shall perform numerical examples to validate the analysis of the fully discrete scheme. We consider a concrete problem \eqref{VtFDEs1} in one-space dimension with $\Omega=(0,1)$ and $T=1$. The uniform mesh size $h=1/M$ for some positive integer $M$.  To illustrate the convergence of the proposed method, we denote the temporal error and the convergence rate as
\begin{equation*}
 E_{2}(\tau,h) = \sqrt{ h\sum_{j=1}^{M-1} \left|U^j_N(\tau,h)-U^j_{2N}(\tau/2,h)\right|^2 }, \quad {\rm rate}^t = \log_{2} \left(\frac{E_{2}(2\tau,h)}{E_{2}(\tau,h)}\right),
\end{equation*}
and the spatial error and the convergence rate as
\begin{equation*}
 G_{2}(\tau,h) = \sqrt{ h\sum_{j=1}^{M-1}  \left|U^j_N(\tau,h)-U^{2j}_{N}(\tau,h/2)\right|^2}, \quad {\rm rate}^x = \log_{2} \left(\frac{G_{2}(\tau,2h)}{G_{2}(\tau,h)}\right).
\end{equation*}

\textbf{Example 1.} We take $\alpha(t)=1-e^{-t}$, $f\equiv 0$ and $u_0(x)=\sin (\pi x)$. We fix $M=32$ to test the temporal convergence rates  and fix $N=64$ to test the spatial convergence rates. We present the numerical results in Table \ref{tab:1}, which illustrates the first-order accuracy in time and the second-order accuracy in space of the scheme \eqref{FEM}, as proved in  Theorem \ref{thm:Error}.

\begin{table}\small\centering
     	\caption{The errors and convergence rates in Example 1.}
     \label{tab:1}  
     {\footnotesize\begin{tabular}{cccccccccccccccccccc}
      \hline\noalign{\smallskip}
           $N$  &  && $E_{2}(\tau,h)$      &  && ${\rm rate}^t$ & $M$  &  && $G_{2}(\tau,h)$      &  && ${\rm rate}^x$ \\
     \noalign{\smallskip}\hline\noalign{\smallskip}
		  128  & && 1.7768e-4         &  &&     *   &8    & && 1.4650e-3         &  &&      *   \\
		  256  & && 9.9362e-5         &  && 0.8385  &16   & && 3.7606e-4         &  && 1.9619   \\
		  512  & && 5.3033e-5         &  && 0.9058   &32   & && 9.4602e-5         &  && 1.9910  \\
         1024  & && 2.7560e-5         &  && 0.9443   & 64   & && 2.3687e-5         &  && 1.9978     \\
         2048  & && 1.4098e-5         &  && 0.9671 &128  & && 5.9240e-6         &  && 1.9994   \\
     \noalign{\smallskip}\hline
     \end{tabular}}
   \end{table}


\textbf{Example 2.} We take $\alpha(t)=\sin(t)$, $f\equiv 0$ and $u_0(x)=x^2(1-x)^2$. The other parameters are chosen as those in Example 1 and numerical results are presented in Table \ref{tab:3}, which leads to the same conclusion as above.
\begin{table}[]\small\centering
     	\caption{The errors and convergence rates in Example 2.}
     \label{tab:3}  
     {\footnotesize\begin{tabular}{cccccccccccccccccccc}
      \hline\noalign{\smallskip}
           $N$  &  && $E_{2}(\tau,h)$      &  && ${\rm rate}^t$ & $M$  &  && $G_{2}(\tau,h)$      &  && ${\rm rate}^x$  \\
     \noalign{\smallskip}\hline\noalign{\smallskip}
		  128  & && 2.1888e-5         &  &&     * &16   & && 2.7669e-5         &  && *     \\
		  256  & && 1.2108e-5         &  && 0.8542 &32   & && 7.2184e-6         &  && 1.9385   \\
		  512  & && 6.4090e-6         &  && 0.9177  &64   & && 1.8234e-6         &  && 1.9851   \\
         1024  & && 3.3111e-6         &  && 0.9528 &128  & && 4.5702e-7         &  && 1.9963      \\
         2048  & && 1.6871e-6         &  && 0.9728  & 256  & && 1.1433e-7         &  && 1.9991 \\
     \noalign{\smallskip}\hline
     \end{tabular} }
   \end{table}

%


\section*{Appendix: Proof of Lemma \ref{lem:Append}}
Since $B(t) u \big|_{t=0} = g( t) * \Delta u\big|_{t=0} = 0$, we commute the operators ${}_0I_t^{1-\varepsilon}$ and $\p_t$ and apply integration by parts  on the right-hand side of the fifth equal sign to obtain
\begin{equation}\label{Append:e1}\begin{array}{l}
{}^R\p_t^\varepsilon \p_t \big(B(t) u\big) = \p_t\, I_t^{1-\varepsilon}\partial_t \big(B(t)u\big)
= \p_t^2\, {}_0I_t^{1-\varepsilon}\big(B(t)u\big) \\[0.05in]
\qquad \ds =\partial_t^2 \!\int_0^t\frac{(t-s)^{-\varepsilon}}{\Gamma(1-\varepsilon)}\int_0^s g( s-y) \Delta u(\bm x,y)  dyds\\[0.15in]
\qquad \ds =\p_t^2 \! \int_0^t \f{\Delta u(\bm x,y)}{\Gamma(1-\varepsilon)} \! \int_y^t (t-s)^{-\varepsilon}g( s-y) ds dy\\[0.15in]
\qquad \ds =\p_t^2\bigg[-\f{\Delta u(\bm x,y)}{\Gamma(1-\varepsilon)} \int_y^t \int_\theta^t (t-s)^{-\varepsilon}g( s-\theta)ds d\theta\bigg|_{y=0}^{y=t}\\[0.15in]
\qquad \ds\quad+\int_0^t \f{\p_y \Delta u(\bm x,y)}{\Gamma(1-\varepsilon)} \int_y^t \int_\theta^t (t-s)^{-\varepsilon}g( s-\theta)dsd\theta dy\bigg] \\[0.15in]
\ds \qquad =\int_0^t \f{\p_y \Delta u(\bm x,y)}{\Gamma(1-\varepsilon)}\, \partial_t^2 \! \int_y^t \int_\theta^t (t-s)^{-\varepsilon}g( s-\theta) dsd\theta  dy\\[0.15in]
\ds \qquad \quad+\f{\Delta u(\bm x,0)}{\Gamma(1-\varepsilon)}\, \partial_t^2 \! \int_0^t \int_\theta^t (t-s)^{-\varepsilon}g( s-\theta)ds d\theta.
\end{array}
\end{equation}
We first prove Case 1. For the convenience of analysis, we rewrite $g( t) =\f{t^{-\alpha( t)}}{\Gamma(1-\alpha( t))} G( t)$ as in \eqref{Lem0:e1}   and the preceding estimates  \eqref{Lem0:e2}--\eqref{Ptg} imply that
 \begin{equation}\label{G:est}
   \ds |G( t)| \le Q |\ln(t)|,\quad | G^\prime( t)| \le Q t^{-1}.
 \end{equation}

To bound \eqref{Append:e1}, it remains to bound the following term for $0\leq y< t$
\begin{equation}\label{Append:e2}\begin{array}{l}
\ds \p_t^2\int_y^t \int_\theta^t (t-s)^{-\varepsilon}g( s-\theta) dsd\theta   \\[0.15in]
\ds = \p_t^2\int_y^t \int_\theta^t (t-s)^{-\varepsilon}\f{(s-\theta)^{-\alpha( s-\theta)}}{\Gamma(1-\alpha( s-\theta))}G( s-\theta) dsd\theta  \\[0.15in]
\ds=\p_t^2 \int_y^t(t-s)^{-\varepsilon} \int_y^s \frac{(s-\theta)^{-\alpha(s-\theta)} }{\Gamma(1-\alpha(s-\theta))} G( s-\theta)d\theta ds\\[0.15in]
\ds=-\p_t^2 \! \int_y^t(t-s)^{-\varepsilon} \bigg [\int_y^s \frac{(s-\theta)^{\alpha(0)-\alpha(s-\theta)}}{\Gamma(1-\alpha( s-\theta))} G( s-\theta) d\frac{(s-\theta)^{1-\alpha( 0)}}{1-\alpha( 0)}\bigg ] ds\\[0.15in]
\ds=\p_t^2 \! \int_y^t (t-s)^{-\varepsilon}\bigg[\frac{(s-y)^{1-\alpha(s-y)}}{\Gamma(1-\alpha(s-y))} G( s-y) \\[0.15in]
\ds~+ \int_y^s (s-\theta) \p_\theta\Big(\frac{(s-\theta)^{-\alpha(s-\theta)}}{\Gamma(1-\alpha(s-\theta))}G( s-\theta)\Big) d\theta\bigg] ds=:\p_t^2 I_1 + \p_t^2 I_2,
\end{array}\end{equation}
where we have used the first estimate in \eqref{Model:e5} to evaluate the following limit
$$
\begin{array}{l} \ds \lim_{\theta \rightarrow s^-} \bigg|\frac{(s-\theta)^{1-\alpha(s-\theta)}}{\Gamma(1-\alpha( s-\theta))} G( s-\theta) \bigg| \le \lim_{\theta \rightarrow s^-}   Q (s-\theta)^{1-\alpha(s-\theta)} |\ln(s-\theta)|\\[0.15in]
\ds \qquad \qquad  =  \lim_{\theta \rightarrow s^-}   Q (s-\theta)|\ln(s-\theta)| =0.
\end{array}
$$
Direct calculations show that $I_2$ in \eqref{Append:e2} can be expressed as
\begin{equation}\label{Append:e3}\begin{array}{rl}
I_2 & \ds =\int_y^t {(t-s)^{-\varepsilon}}\int_y^s (s-\theta)^{1-\alpha(s-\theta)} \mathcal K(\theta,s)d\theta ds
\end{array}
\end{equation}
with
\begin{equation}\label{Append:K}\begin{array}{rl}
\hspace{-0.1in}\ds \mathcal K(\theta,s) & \hspace{-0.1in}\ds := \frac{1}{\Gamma(1-\alpha(s-\theta))} \bigg[G( s-\theta)\Big(-\p_\theta\alpha(s-\theta)\ln(s-\theta)\\[0.15in]
& \ds +\frac{\alpha(s-\theta)}{s-\theta}+\frac{\Gamma'(1-\alpha(s-\theta))\p_\theta\alpha(s-\theta)}{\Gamma(1-\alpha(s-\theta))}\Big) + \p_\theta G( s-\theta)\bigg].
\end{array}\end{equation}
We combine the estimates in  \eqref{G:est} with the first estimate in \eqref{Model:e5} to bound the integrand of the inner integral in \eqref{Append:e3} by
\begin{equation*}\begin{array}{l}
\ds \big|(s-\theta)^{1-\alpha(s-\theta)} \mathcal K(\theta,s)\big| \le Q (s-\theta)^{1-\alpha(s-\theta)} \big( \ln^2(s-\theta) + (s-\theta)^{-1}\big)\\[0.1in]
\ds \qquad   \qquad  \qquad \qquad \qquad  \qquad\le Q (s-\theta)^{-\alpha(s-\theta)} \le Q.
  \end{array}
\end{equation*}
 In addition, the estimate \eqref{Model:e5} and similar derivations as \eqref{Lem0:e1}-\eqref{Lem0:e2} combined with \eqref{Append:K} imply that
\begin{equation}\label{Append:K:est}
  \Big|\p_s\big[(s-\theta)^{1-\alpha(s-\theta)} \mathcal K(\theta,s)\big]\Big| \le Q (s-\theta)^{-\alpha(s-\theta)} \ln^2(s-\theta) \le Q(s-\theta)^{-\varepsilon}.
\end{equation}
We then invoke the above two estimates to  apply integration by parts and incorporate the preceding limit in \eqref{Lop:e1} to get
\begin{equation}\label{Append:e4}\begin{array}{rl}
\ds  \hspace{-0.15in}I_2 & \hspace{-0.125in}\ds = \int_y^t\frac{(t-s)^{1-\varepsilon}}{1-\varepsilon}\p_s\int_y^s (s-\theta)^{1-\alpha(s-\theta)} \mathcal K(\theta,s)d\theta ds \\[0.15in]
 \hspace{-0.125in}&\hspace{-0.125in} \ds =Q \int_y^t\f{(t-s)^{1-\varepsilon}}{1-\varepsilon}ds + \int_y^t\frac{(t-s)^{1-\varepsilon}}{1-\varepsilon}\int_y^s \p_s \big [ (s-\theta)^{1-\alpha(s-\theta)} \mathcal K(\theta,s)\big]d\theta ds\\[0.15in]
 \hspace{-0.125in} &\hspace{-0.125in}\ds = Q \f{(t-y)^{2-\varepsilon}}{(1-\varepsilon)(2-\varepsilon)} \!+\! \int_y^t\frac{(t-s)^{2-\varepsilon}}{(1-\varepsilon)(2-\varepsilon)}\p_s\int_y^s\p_s\big[ (s-\theta)^{1-\alpha(s-\theta)}\mathcal K(\theta,s)\big]d\theta ds.
\end{array}
\end{equation}
Differentiate \eqref{Append:e4} twice  and incorporate the estimate \eqref{Append:K:est} to bound
\begin{equation}\label{Append:e5}\begin{array}{rl}
\hspace{-0.1in}\ds | \p_t^2 I_2 | &\hspace{-0.1in}\ds \le Q (t-y)^{-\varepsilon} + \bigg|\int_y^t (t-s)^{-\varepsilon}\p_s\int_y^s\p_s\big[ (s-\theta)^{1-\alpha(\theta,s-\theta)}\mathcal K(\theta,s)\big]d\theta ds\bigg|\\[0.125in]
&\hspace{-0.1in} \ds\leq Q  (t-y)^{-\varepsilon} + Q\int_y^t(t-s)^{-\varepsilon}(s-y)^{-\varepsilon}ds\leq Q(t-y)^{-\varepsilon}.
\end{array}\end{equation}
To bound $I_1$ in \eqref{Append:e2}, we utilize the relation
$$(s-y)^{1-\alpha( s-y)} = \int_y^s \p_z\big[(s-y)^{1-\alpha( z-y)}\big] dz + (s-y)^{1-\alpha( 0)}$$
to rewrite $I_1$ as
\begin{equation*}\begin{array}{l}
\ds I_1 = \int_y^t \frac{(t-s)^{-\varepsilon}(s-y)^{1-\alpha(0)}}{\Gamma(1-\alpha(s-y))} G( s-y)ds\\[0.15in]
\ds\quad\quad+\int_y^t (t-s)^{-\varepsilon}\int_y^s\p_z\big[(s-y)^{1-\alpha( z-y)}\big] G( s-y) dzds=:I_{1a}+I_{1b},
\end{array}\end{equation*}
where $\p_t^2 I_{1b}$ can be bounded in a similar manner to that of $\p_t^2 I_2$.
We employ the variable substitution introduced in \eqref{Lem1vo:sub} to reformulate  $ I_{1a}$  and then employ the first estimate in \eqref{G:est} to bound the dominant term of $\p_t^2 I_{1a}$ as
\begin{equation}\label{I1a}\begin{array}{l}
\ds | \p_t^2 I_{1a}| = \bigg|\p_t ^2 \bigg[ (t-y)^{2-\varepsilon -\alpha( 0)} \int_0^1 \f{(1-z)^{-\varepsilon} z^{1-\alpha( 0)} G( (t-y)z)}{\Gamma(1-\alpha( (t-y)z))} dz\bigg]\bigg|\\[0.15in]
\ds  \qquad \qquad  \le  Q (t-y)^{-\varepsilon }\int_0^1 (1-z)^{-\varepsilon} z  \big|\ln[(t-y)z]\big| dz\\[0.15in]
\ds  \qquad \qquad  \le  Q (t-y)^{-\varepsilon } |\ln(t-y)|\int_0^1 (1-z)^{-\varepsilon} z^{1-\varepsilon} dz\\[0.15in]
\ds  \qquad \qquad  \le  Q (t-y)^{-\varepsilon } |\ln(t-y)|B(1-\varepsilon,2-\varepsilon) \le Q (t-y)^{-\varepsilon } |\ln(t-y)|.
\end{array}
\end{equation}

 We incorporate \eqref{Append:e5} and  the estimates for $\p_t^2 I_1$ into \eqref{Append:e2},  and combine the resulting estimate with \eqref{Append:e1} and the fact that $\Delta u(\bm x, 0) = \Delta u_0$  to prove \eqref{lem:Append:e1} for $0 < \varepsilon \ll 1$. The estimate \eqref{lem:Append:e1} with $\varepsilon = 0$ can be obtained by letting $\varepsilon \downarrow 0$ in \eqref{lem:Append:e1}.

 We note that for Case 2, the estimate \eqref{I1a}, which is the most singular estimate among all the estimates for $\p_t^2 I$,  is augmented to $| \p_t^2 I_{1a}| \le Q (t-y)^{-\varepsilon }$ since $|G( t)| \le Q $ by \eqref{Lem0:e1}-\eqref{Lem0:e2} and  \eqref{Lem0:lop1}. We incorporate  \eqref{Append:e5} and  the modified estimates for $\p_t^2 I_1$ into \eqref{Append:e2},  and combine the resulting estimate with \eqref{Append:e1} to prove \eqref{lem:Append:e2} for $0 < \varepsilon \ll 1$. We follow the preceding procedures to prove the estimate \eqref{lem:Append:e2} with $\varepsilon = 0$.

\section*{Acknowledgments}
This work was partially supported by the National Natural Science Foundation of China (No. 12301555), the National Key R\&D Program of China (No. 2023YFA1008903), and the Taishan Scholars Program of Shandong Province (No. tsqn202306083).

\end{document}